\pdfoutput=1
\documentclass[11pt]{amsart}

\usepackage{epsfig}
\usepackage{stmaryrd}
\usepackage{graphicx}
\usepackage{amscd}
\usepackage{amsmath}
\usepackage{amsfonts}
\usepackage{mathrsfs}
\usepackage{mathabx}
\usepackage{amssymb}
\usepackage{verbatim}
\usepackage{comment}
\usepackage{color}

\newtheorem{theorem}{Theorem}[section]
\theoremstyle{plain}

\newtheorem{corollary}[theorem]{Corollary}

\newtheorem{defi}[theorem]{Definition}

\newtheorem{lemma}[theorem]{Lemma}

\newtheorem{prop}[theorem]{Proposition}
\newtheorem{remark}[theorem]{Remark}

\numberwithin{equation}{section}

\def\dprod{\,{\scriptscriptstyle\bullet}\,}
\def\nuhat{\widehat{\nu}}
\def\muhat{\widehat{\mu}}

\def\dist{{\rm dist}}

\def\Hk{{\mathcal H}}

\def\Ok{{\mathcal O}}
\def\Nk{{\mathcal N}}
\def\Dh{\dim_{\rm H}}

\def\bzeta{{\mathbf \zeta}}

\def\half{\frac{1}{2}}

\newcommand{\lam}{\lambda}

\newcommand{\gam}{\gamma}

\newcommand{\sig}{\sigma}

\def\b0{{\bf 0}}

\newcommand{\R}{{\mathbb R}}
\newcommand{\Q}{{\mathbb Q}}
\newcommand{\Z}{{\mathbb Z}}
\newcommand{\C}{{\mathbb C}}

\def\N{{\mathbb N}}

\def\bp{{\bf p}}

\def\bz{{\bf z}}
\def\wh{\widehat}

\def\Lk{{\mathcal L}}
\def\Mk{{\mathcal M}}

\def\Dk{{\mathcal D}}

\def\T{{\mathbb T}}

\def\bx{{\mathbf x}}
\def\be{{\mathbf e}}

\def\beq{\begin{equation}}
\def\eeq{\end{equation}}
\def\wt{\widetilde}
\newcommand{\Ek}{{\mathcal E}}

\newcommand{\eps}{{\varepsilon}}

\def\ov{\overline}

\def\Span{{\rm Span}}

\def\ve1{\vec{1}}

\def\Tk{{\mathcal T}}
\def\Ak{{\mathcal A}}

\def\pb{\boldsymbol{p}}

\def\balpha{\boldsymbol{\alpha}}
\def\thb{\boldsymbol{\theta}}

\def\Dsc{{\mathscr D}}

\def\what{\widehat}
\def\re{{\rm Re}}
\def\im{{\rm Im}}
\def\thet{\vartheta}

\def\scr{\scriptscriptstyle \R/\Z}

\def\nula{\nu_\lam}

\def\ben{\begin{enumerate}}
\def\een{\end{enumerate}}

\begin{document}

\title[Fourier decay]
{Fourier decay and absolute continuity 
for typical homogeneous  self-similar measures in $\R^d$ for $d\ge 3$}

\author[Boris Solomyak]{Boris Solomyak}

\address{Boris Solomyak, Department of Mathematics, Bar-Ilan University, Ramat Gan, 5290002 Israel}
\email{bsolom3@gmail.com}

\thanks{Supported in part by the Israel Science Foundation grant \#1647/23.}

\date{\today}

\begin{abstract}
We consider  iterated function systems (IFS) in $\R^d$ for $d\ge 3$ 
of the form $\{f_j(x) = \lam \Ok x  + a_j\}_{j=0}^m$, with $a_0=0$ and $m\ge 1$. Here $\lam\in (0,1)$ is the contraction ratio and
$\Ok$ is an orthogonal matrix. Given a positive probability vector $\pb$, there is a unique invariant (stationary) measure for the IFS, called (in this case) a homogeneous self-similar measure, which we denote $\mu(\lam \Ok, \Dk, \pb)$, where $\Dk = \{a_0,\ldots,a_m\}$ is the set of ``vector digits''.
We obtain two  results on Fourier decay for such measures. First we show that if $\Dk$ spans $\R^d$,  then for every fixed $\Ok$ and
$\pb$ the measure $\mu(\lam \Ok, \Dk, \pb)$
has power Fourier decay (equivalently, positive Fourier dimension) for all but a zero-Hausdorff dimension set of $\lam$. In our second result we
do not impose any restrictions on $\Dk$, other than the necessary one of affine irreducibility, and obtain power Fourier decay for almost all
homogeneous self-similar measures; however, only for even $d\ge 4$. 
Combined with recent work of Corso and Shmerkin \cite{CoShm24}, these results imply absolute continuity for almost all self-similar
measures under the same assumptions, in the super-critical parameter region.
\end{abstract}

\maketitle

\thispagestyle{empty}

\section{Introduction and statement of results}

For a finite positive Borel measure $\mu$ on $\R^d$, consider the Fourier transform
$$
\wh\mu(\xi) = \int_{\R^d} e^{-2\pi i \langle \xi, x\rangle}\,d\mu(x).
$$
The measure is called {\em Rajchman} if $\wh\mu(\xi)$ tends to zero at infinity. By the Riemann-Lebesgue Lemma, every absolutely continuous measure
is Rajchman, but absolute continuity is not necessary. It is a delicate and much studied question to decide which singular measures are Rajchman.
Assuming they are, it is important to obtain specific quantitative bounds for the rate of decay at infinity. 
Let $\Mk(\R^d)$ be the set of finite positive measures in $\R^d$
For $\gam>0$ consider
$$
\Dsc_d(\gam) = \bigl\{\nu\in \Mk(\R^d):\ \left|\widehat{\nu}(\xi)\right| = O_\nu(|\xi|^{-\gam}),\ \ |\xi|\to \infty\bigr\},
$$
and denote $\Dsc_d= \bigcup_{\gam>0} \Dsc_d(\gam)$. A measure $\nu$ is said to have {\em power Fourier decay} if $\nu\in \Dsc_d$.
This is equivalent to having a positive {\em Fourier dimension}, see \cite{Mattila:Fourier-book}.
Further, let
$$
\Dsc_{d,\log} = \bigl\{\nu\in \Mk(\R^d):\ \exists\,\alpha>0,\ \left|\widehat{\nu}(\xi)\right| = O_\nu\bigl(\bigl|\log|\xi|\bigr|^{-\alpha}\bigr),\ \ |\xi|\to \infty\bigr\}.
$$
Such measures are said to have {\em polylogarithmic Fourier decay}. Of course, other decay rates may occur as well.

There is on-going very intensive research of the Fourier decay phenomena and its applications, for
a broad class of measures which may be termed ``fractal''. We are not going to give a comprehensive literature review: a ``sample'' of recent papers, which includes two surveys, is \cite{ARW23,ARW24,Sahlsten23,BaBa24,BaKhaSa24,LPS25}; many more references can be found there. 
However, much of the recent progress is for measures coming from non-linear dynamics, in particular, for stationary (invariant) measure of non-linear iterated function systems (IFS). Here we focus exclusively on the linear case.

A measure $\mu$ is called {\em self-affine}  if it is the invariant measure for a self-affine IFS $\{f_j\}_{j=0}^m$, with $m\ge 1$, where $f_j(x) = A_jx + a_j$, the matrices $A_j:\R^d\to \R^d$ are invertible linear contractions (in some norm) and $a_j\in \R^d$ are   ``digit'' vectors.
This means that for some probability vector $\pb = (p_j)_{j\le m}$ holds
\beq \label{ssm1}
\mu = \sum_{j=0}^m p_j (\mu\circ f_j^{-1}).
\eeq
The IFS is {\em self-similar} if all $A_j$ are contracting similitudes, that is, $A_j = \lam_j \Ok_j$ for some $\lam_j\in (0,1)$ and orthogonal matrices $\Ok_j$. 
The IFS and the corresponding invariant measure are called {\em homogeneous self-affine} (respectively, {\em homogeneous self-similar}) 
if $A_j=A$ for all $j\le m$, and $A$ is a general linear contracting map (respectively, a contracting similitude). Our  results are on self-similar homogeneous IFS, and we
review the relevant history with some detail; however, we do not intend to duplicate the excellent survey by Sahlsten \cite{Sahlsten23}, which is much more exhaustive.
Our main motivation for showing the power Fourier decay is to study {\em absolute continuity} of self-similar measures, but 
it has many other applications which we do not discuss here, see, e.g.,  \cite{Sahlsten23}.

\subsection{Historical background} We start with the most basic case: $d=1$, $m=2$, a homogeneous self-similar measure. Up to an affine-linear conjugacy, we obtain a 1-parameter family of IFS $\{\lam x, \lam x+1\}$. The resulting self-similar measures are known as {\em Bernoulli convolutions}; they were
extensively studied since the 1930's. Denote by $\nula^p$ the Bernoulli convolution, corresponding to the probabilities $(p,1-p)$. In the ``classical case''
$p=1/2$ we omit the superscript $p$.
Erd\H{o}s \cite{Erd1} proved that $\nu_\lam$ is not Rajchman when $\theta=1/\lam$ is a {\em PV (or Pisot)} number, that is, 
an algebraic integer greater than one, all of whose Galois conjugates lie inside the unit circle. Salem \cite{Salem} showed that if
$1/\lam$ is not a Pisot number, then  $\nuhat_\lam$ does vanish at infinity. What about decay rates?
Erd\H{o}s \cite{Erd2}  proved that for any compact interval $J\subset (0,1)$ there exists $\alpha=\alpha(J)>0$ such that
$\nu_\lam\in \Dsc_1(\alpha)$ for a.e.\ $\lam\in J$. In a short but illuminating expository note, Kahane \cite{kahane} indicated that Erd\H{o}s' argument actually gives that $\nu_\lam\in \Dsc_1$ for all $\lam\in (0,1)$ in the complement  of an exceptional set of zero Hausdorff dimension. (This technique became known as the {\em Erd\H{o}s-Kahane argument}, which turned out to be useful in other contexts as well.)
All the above results hold for any
$p\in (0,1)$; moreover,
 the Erd\H{o}s-Kahane argument works for any family of homogeneous IFS on the line, with the contraction rate as a parameter. As it often happens with results of ``almost sure'' kind, few specific $\lam$ are known, all algebraic, for which $\nu_\lam$ has power Fourier decay.
Dai, Feng, and Wang \cite{DFW} proved this for the classical Bernoulli convolution $\nula$, when $\theta=\lam^{-1}$ is a  {\em Garsia number}, i.e.,
an algebraic integer, all of whose conjugates have modulus greater than 1, and the constant term of its minimal polynomial is $\pm 2$.
Recently, Streck \cite{Streck23} generalized this to a larger class of homogeneous IFS, still with a very special
algebraic contraction ratio.

What about weaker decay rates for explicit $\lam$? Kershner \cite{Kershner} (see also \cite{Dai}) proved that $\nula\in \Dsc_{1,\log}$ for a rational
contraction rate $\lam = p/q$, with $1  < p< q$ (note that for $\lam = 1/q$, with $q\ge 3$, the measure $\nula$ is non-Rajchman; in fact, these are the
simplest Pisot Bernoulli convolutions). In a joint work with Bufetov \cite[Corollary A.3]{BuSo14} it was shown that if $\lam^{-1}$ is an 
algebraic integer which has at least one conjugate outside the unit circle, then $\nula\in \Dsc_{1,\log}$ as well; this was extended by Gao and Ma \cite{GaoMa17} to an arbitrary set of digits.

The case of a {\em Salem number} $\theta = \lam^{-1}$ is very interesting, but remains mysterious. This is, by definition, an algebraic integer which has all of its conjugates satisfying $|\theta_j|\le 1$, but not a Pisot number.
It is known that $\nula\notin \Dsc_1$ in the Salem case, see
\cite{kahane} (and \cite[Section 5]{PSS00} for more details).
Curiously, it is not known whether the property $\nula\in \Dsc_1(\alpha)$ is 
topologically generic (of second category) in $(0,1)$. In fact, if it is, then the Salem numbers do not accumulate to one (which would confirm a
well-known conjecture). This follows from the fact that the set $\{\lam\in (0,1):\ \nula\notin \Dsc_1\}$ is $G_\delta$, see \cite[Lemma 5.4]{PSS00}, and
the implication $\nula\notin \Dsc_1 \implies \ \nu_{\lam^2} \notin \Dsc_1$. 

 Note that all the delicate arithmetic phenomena, having to do with the contraction ratio $\lam$, disappear if we introduce dependence on a translation parameter. For instance, for the family of IFS $\{\lam x, 1+\lam x, a+\lam x\}$, with any fixed $\lam\in (0,1)$ and any probability vector
$\pb>0$, the corresponding self-similar measure belongs to $\Dsc_1$ for all $a\in \R\setminus E$, where $\Dh(E)=0$; see 
\cite[Proposition 3.1]{ShSol16}.

Next we briefly discuss absolute continuity of Bernoulli convolutions and more general homogeneous self-similar measures on $\R$.		
Using the convolution structure of $\nula$ and his ``a.e.\ $\lam$'' power Fourier decay result, Erd\H{o}s \cite{Erd2} deduced that $\nula$
is absolutely continuous, even of arbitrarily high smoothness, fixed
in advance, for Lebesgue-a.e.\ $\lam$ in a sufficiently small neighborhood of 1. A different, {\em transversality method}, was used to confirm
absolute continuity of $\nula$ for Lebesgue-a.e.\ $\lam\in (1/2,1)$ in \cite{Solomyak95} (see also \cite{PeresSolomyak96}), and to estimate the Hausdorff dimension of the exceptional set \cite{PS00}. More recently, following breakthrough results of Hochman \cite{Hochman14} on the dimension of
self-similar measures, Shmerkin \cite{Shmerkin14}
showed that this exceptional set has Hausdorff dimension zero, invoking the classical Erd\H{o}s-Kahane statement on power Fourier decay and a 
beautiful ``convolution lemma,'' which we also use in this paper:

\begin{lemma}[{Shmerkin \cite{Shmerkin14}}] \label{lem:Shm}
Suppose that $\mu$ and $\nu$ are finite measures in $\R^d$, such that $\Dh(\mu) = d$ and $\nu\in \Dsc_d$. Then $\mu*\nu$ is absolutely continuous with
respect to $\Lk^d$.
\end{lemma}

Many of the above results, especially those for ``typical'' parameters, extend to arbitrary homogeneous self-similar measures on $\R$, with various kinds
of parameter dependence, see \cite{Shmerkin14,ShSol16}. 
Until a few years ago,
the only explicit parameters $\lam$, for which $\nula$ is known to be absolutely continuous, were the reciprocals of Garsia numbers.  More recently, 
in a continuation of the line of research pioneered by Hochman \cite{Hochman14}, but with many further delicate and technical innovations, 
absolutely continuity was established
 for a much larger class of examples,
all algebraic and satisfying certain arithmetic conditions, due to Varj\'u \cite{Varju1} and Kittle \cite{Kittle}.

\smallskip

In higher dimensions  issues of irreducibility arise. Following \cite{Hochman}, we say that the IFS is {\em affinely irreducible} if the attractor is not contained in a proper affine subspace of $\R^d$. It is easy to see that this is a necessary condition for the self-affine measure to be Rajchman, so this will always be our assumption. The IFS is called {\em linearly irreducible} if the linear maps $A_j$ do not have a common proper invariant subspace.

We next consider self-similar homogeneous IFS in $\R^2$. Affine irreducibility means that the attractor is not contained in a line.
There are 2 cases:

\begin{enumerate}
\item {\em Linearly reducible case.}
If $A_j = A = \lam\cdot I$, for $\lam\in (-1,1)$, then power Fourier decay, for all but a zero Hausdorff dimension set of exceptions,
 easily follows by the classical 1-dimensional argument. However, to deduce ``typical'' 
absolute continuity via Shmerkin's Lemma \ref{lem:Shm} is problematic, except in some very
special symmetric cases, because Hochman's higher-dimensional results
\cite{Hochman} require linear irreducibility. These issues will be clarified below for self-similar homogeneous IFS in $\R^d$ for any $d\ge 2$.

\smallskip

\item  {\em Linearly irreducible case.}
View $\R^2\cong \C$ as the complex plane, and let $Az = \lam z$, where $\lam$ is a non-real complex number, with $|\lam|<1$. In this setting, power Fourier
decay was obtained for all $\lam$ outside a zero-Hausdorff dimensional set of exceptional 
parameters in a joint work with Shmerkin \cite{ShSol16b}. Moreover, Hochman's  \cite[Theorem 1.5]{Hochman} applies here, so we obtained absolute 
continuity of the self-similar measure for a.e.\ parameter in the super-critical region, more precisely, outside of a (different) 
zero-Hausdorff dimensional set of parameters, see \cite[Theorem B]{ShSol16b}.
\end{enumerate}

The next step is to consider homogeneous self-similar IFS in dimensions $\ge 3$, and this is the topic of the current paper. One of the difficulties is that such
a system is {\em never} linearly irreducible.

\smallskip

We only mention the case of {\em non-homogeneous self-similar IFS} in passing, without going into details.
For  questions on  Fourier decay  of non-homogeneous self-similar measures on $\R$, the reader is referred to 
\cite{AHW21,Bremont,LS1,Sol_Fourier, VarYu} and references therein. Absolute continuity of such measures, for almost all relevant parameters, was established in \cite{SSS}, and this was
extended to the planar case in a joint work with \'Spiewak \cite{SoSp23}. Lindenstrauss and Varj\'u \cite{LinVar16} proved that in dimension $d\ge 3$, if  the rotation parts of the similarities generate random walk with a spectral gap, then the self-similar measure is absolutely continuous, assuming that all the 
contraction rations are sufficiently close to one. In the survey \cite{Sahlsten23} Sahlsten showed how the methods of \cite{LinVar16}
can be used to obtain power Fourier decay in their setting (e.g., for algebraic rotation parts of the similarities in $\R^d$, $d\ge 3$, when the spectral gap hypothesis holds)  for any common contraction ratio $0 < \lam <1$, see \cite[Theorem 4.3]{Sahlsten23}.
Recently a significant
advance was achieved by Kittle and Kogler \cite{KiKog24}, including 
first explicit examples of absolutely continuous non-homogeneous self-similar measures in dimensions 1 and 2, and an extension of \cite{LinVar16}.

\subsection{Statement of new results}
Denote the digit set by $\Dk:=\{a_0,\ldots,a_m\}$.
By a conjugation with a translation, we can always assume that $a_0=0\in \Dk$. In this case 
affine irreducibility is equivalent to the digit set $\Dk$  being a {\em cyclic family} for $A$, that is, $\R^d$ being the smallest $A$-invariant subspace containing $\Dk$.
Denote the homogeneous self-similar measure, determined by the contraction similarity map $A$, digit set $\Dk$, and the
probability vector $\pb$, by $\mu(A,\Dk,\pb)$. We will write $\pb>0$ if all $p_j>0$.
The special feature of a homogeneous self-similar (and a self-affine) measure is that it can be expressed as an infinite convolution product
\beq \label{eq-conv}
\mu(A,\Dk,\pb) = \Bigl(\Asterisk\prod\limits_{n=0}^\infty\Bigr) \sum_{j=0}^m p_j \delta_{A^n a_j}.
\eeq
For every $\pb>0$ it is supported on the attractor (self-similar set) 
$$
K_{A,\Dk}:= \Bigl\{x\in \R^d: \ x = \sum_{n=0}^\infty A^n b_n,\ b_n \in \Dk\Bigr\}.
$$

\begin{theorem} \label{th:new1}
For $d\ge 1$, let $A = \thet^{-1} \Ok$, where $\thet>1$ and $\Ok\in O(d,\R)$. Let $\Dk = \{0,a_1,\ldots,a_m\}\subset \R^d$ be any set of digits in $\R^d$, such that $\Span(\Dk) = \R^d$ (so that $m\ge d$). Let $\pb$ be a positive probability vector. 
Then the corresponding self-similar measure $\mu(\thet^{-1}\Ok, \Dk, \pb)$ belongs to $\Dsc_d$ for all $\thet>1$ outside of a set of Hausdorff dimension zero.
\end{theorem}

The proof is based on a variant of the Erd\H{o}s-Kahane technique.

\smallskip

Next we state our second result on  Fourier decay. 
The natural measure on the set of
similitudes $A =  \thet^{-1}\cdot \Ok$, with $\thet > 1$ and
$\Ok \in O(d,\R)$ is
is $\Lk^1|_{(0,1)}\times m_d$, where $m_d$ is the Haar measure on $O(d,\R)$.
Every orthogonal matrix
is diagonalizable, with eigenvalues of modulus 1. Having a real eigenvalue $\pm 1$ is a degenerate case for $d$ even; for a.e.\ orthogonal matrix 
in $\R^{2s}$ all the eigenvalues are non-real and distinct. If $d=2s+1$, then $m_d$-a.e.\ orthogonal matrix in $O(d,\R)$ has one real eigenvalue $\pm 1$
and $s$ non-real distinct eigenvalues. Denote $\T^+ :=\{e^{2\pi i \alpha}:\ \alpha\in (0,\half)\}$. We will use the notation
$[N] = \{1,\ldots,N\}$ for $N\in \N$.

\begin{theorem} \label{thm2}
For almost every contraction similitude $A=\thet^{-1}\Ok$ in $\R^d$ for an even $d$, for every cyclic set of digits $\Dk$ and a probability vector $\pb>0$, the
homogeneous self-similar measure $\mu= \mu(A, \Dk, \pb)$ in $\R^d$ has power Fourier decay, i.e.,\ $\mu\in \Dsc_d$.

More precisely, for $d=2s$, with $s\ge 1$, let $\Ok$ be an orthogonal matrix with distinct non-real eigenvalues
$e^{\pm 2\pi i \alpha_j}$. Using a linear conjugacy and the identification $\R^{2s} \cong \C^s$, we can assume 
without loss of generality, that $\Ok$ is a diagonal 
complex $s\times s$ matrix
$$
\Ok = {\rm Diag}[e^{-2\pi i \alpha_1},\ldots, e^{-2\pi i \alpha_s}],\ \ \alpha_j \in (0,1/2).
$$
Suppose that the digit set $\Dk=\{0,a_1,\ldots,a_m\}$, with $a_j = (a_j^{(1)}, \ldots, a_j^{(s)})\in \C^s$, is a cyclic set for $A$
(equivalently,  for every $\ell\in [s]$ there exists $a_j$
with $a_j^{(\ell)}\ne 0$). Fix $\thet>1$ and $\pb>0$, and
 let $\ell\in [s]$. Then for every list $\{\alpha_j\}_{j\in [s]\setminus \ell}$ of distinct numbers in $(0,1/2)$, there exists a set 
 $\Ek_\ell\subset \thet\cdot\T^+$ of
zero Hausdorff dimension, such that for all $\alpha_\ell$, with $e^{-2\pi i\alpha_\ell}\notin \thet^{-1}\cdot \Ek_\ell$,  the Fourier transform  $\what\mu(\thet^{-1}\Ok, \Dk, \pb)$ 
has uniform power decay in the direction of vectors $\xi\in \R^{2s}$ satisfying 
$$
|\xi_{2\ell-1} + i\xi_{2\ell}| = \max_{j\le s} |\xi_{2j-1} + i\xi_{2j}|.
$$
\end{theorem}

\begin{remark} \label{rem1} {\em 
(a) The first claim in the theorem above, about almost every contraction similitude $A=\thet^{-1}\Ok$, follows from the more precise statement by Fubini's Theorem, with the exceptional set of parameters being $\bigcup_{\ell\in [s]} \bigl\{\balpha\in (0,\half)^\ell:\ e^{-2\pi i \alpha_\ell}\in \thet^{-1}\cdot \Ek_\ell\bigr\}$.

(b)
The restriction to only even dimensions $d$ is an artefact of the proof; we do not think it is essential, but we were not able to remove it.
It can be explained, roughly, as follows: 

The argument proceeds by fixing all the eigenvalues, except one, say, $\ell$'s eigenvalue, and running a variant of the Erd\H{o}s-Kahane scheme to estimate the Hausdorff dimension of the set of ``bad'' values of the undetermined eigenvalue for the Fourier decay. The problem arises 
for direction vectors with a small $\ell$-component, with the bounds becoming non-uniform. Thus we are forced to exclude {\em all} 
eigenvalues in turn.
Suppose $d=2s+1$.
Almost every similitude in $\R^{2s+1}$ has $2s$ non-real distinct eigenvalues and one real eigenvalue.
Excluding the real eigenvalue does not work, since the other eigenvalues determine it up to sign. All that we can get
for a typical homogeneous self-similar measure in $\R^{2s+1}$, with a single real eigenvalue,
is uniform power Fourier decay outside of an arbitrary circular cone around the real eigen-direction, fixed in advance. This is shown
 by a minor modification of the proof of Theorem~\ref{thm2}; we leave the details to the reader.
}
\end{remark}


\subsection{Application: fat Sierpi\'nski tetrahedron}
Consider the following IFS in $\R^3$: $$\{\lam x,\ \lam x + \be_1,\ \lam x + \be_2,\ \lam x + \be_3\}.$$ We call its attractor, denoted 
$S^{(3)}_\lam$, 
for $\lam\in (\half,1)$, the ``fat Sierpi\'nski tetrahedron,'' by analogy with the ``fat Sierpi\'nski gasket,'' studied by many authors;
see, e.g., \cite{SiSol02,BroMonSid04,Jordan05,JorPoli06}.

\begin{corollary}
For all $\lam\in (4^{-1/3},1)$, 
outside of a set of zero Hausdorff dimension, the set $S_\lam$ has positive 3-dimensional Lebesgue measure.
\end{corollary}

\begin{proof}[Proof sketch]
We use the method of Shmerkin \cite{Shmerkin14} to show absolute continuity of the uniform self-similar measure on 
$S^{(3)}_\lam$ for $\lam\in 
(4^{-1/3},1)\setminus \Ek$, where $\Dh(\Ek)=0$.  We can replace the set of digits $\{0,\be_1,\be_2, \be_3\}$ by the set of vertices of a regular tetrahedron $\Tk\subset \R^3$, which has a larger group of symmetry, via a linear conjugacy, which preserves all the properties
of the self-similar measure.
Denote by $\mu_\lam$ the corresponding uniform self-similar measure in $\R^3$.
We can write 
\beq \label{eq:split1}
\mu_\lam = \mu_{\lam^k} * \wt \mu_\lam^{(k)},
\eeq
where $\wt\mu_\lam^{(k)}$ is obtained by skipping every $k$-th term in the random base-$\lam$ series representation. 
 It suffices to obtain the absolute continuity result
for the intervals $\bigl(4^{-\frac{k-1}{3k}},1\bigr)$  for all $k$. We have
$\mu_{\lam^k}\in \Dsc_3$ for all
$\lam \in (0,1) \setminus \Ek_k^{(1)}$, with $\Dh(\Ek_k^{(1)})=0$, by Theorem~\ref{th:new1}.

For the measure $\wt\mu_\lam^{(k)}$ we would like to use Hochman's results \cite{Hochman} on the dimension of self-similar measures
in $\R^d$.
Note that 
$\wt\mu_\lam^{(k)}$ is a self-similar measure of similarity dimension $\frac{\log(4^{k-1})}{\log(1/\lam^k)}>3$ in this parameter interval.
We cannot use \cite[Corollary 1.6]{Hochman} directly, since the homogeneous IFS in $\R^3$ is not linearly irreducible.
However, in view of the symmetry, since $\wt\mu_\lam^{(k)}$ has uniform weights, 
we obtain the same self-similar measure if the functions mapping 
the big tetrahedron $\Tk$ onto the $4^{k-1}$ small ones, are pre-composed with a rotation taking  $\Tk$ into itself. This rotation may be
different for each smaller tetrahedron. This allows us to avoid a common proper invariant subspace for the linear parts of the IFS maps;
for instance, by using rotations about lines perpendicular to the facets of $\Tk$. A standard argument (see \cite[Section 6.6]{Hochman}) yields that there exists a set $\Ek_k^{(2)}$ of zero Hausdorff (even packing) dimension, such that for all 
$\lam\in \bigl(4^{-\frac{k-1}{3k}},1\bigr)\setminus \Ek_k^{(2)}$ the resulting IFS satisfies the exponential separation condition, hence
$\dim_H(\wt \mu_\lam^{(k)}) = 3$ for all $\lam\in \bigl(4^{-\frac{k-1}{3k}},1\bigr)\setminus \Ek_k^{(2)}$ by 
Hochman's \cite[Corollary 1.6]{Hochman}.
Finally, using (the higher-dimensional version of) \cite[Lemma 2.1]{Shmerkin14}, we
obtain absolute continuity for all $\lam\in \bigl(4^{-\frac{k-1}{3k}},1\bigr)\setminus (\Ek_k^{(1)}\cup \Ek_k^{(2)})$, as desired.
\end{proof}

The ``fat tetrahedron'' construction and the Lebesgue measure claim 
can be analogously extended to an arbitrary dimension $d\ge 4$; we used $d=3$ for 
simplicity.

\medskip

Recall Hochman's exponential separation (ES) condition, used above, in the special case of self-similar IFS.

\begin{defi}[M. Hochman \cite{Hochman14,Hochman}] {\em
A homogeneous IFS $f_j(x) = Ax + a_j,\ a_j\in \Dk$, is said to satisfy the ES condition if there exist a sequence $N_j\to \infty$ and $\eps>0$,
such that for any two words $v\ne w\in \Dk^{N_j+1}$, holds
$$
\Bigl\|\sum_{n=0}^{N_j} A^n(v_n - w_n)\Bigr\|\ge \eps^{N_j}.
$$
Equivalently, for any positive probability vector $\pb$, the atoms of discrete measures approximating the self-similar measure $\mu(A,\Dk,\pb)$:
$$
\mu_{N_j} (A,\Dk,\pb) := \Bigl(\Asterisk\prod\limits_{n=0}^{N_j}\Bigr) \sum_{j=0}^m p_j \delta_{A^n a_j},
$$
are $\eps^{N_j}$-separated.}
\end{defi}

\subsection{Application: absolute continuity via $L^q$ dimension}

This application is based on a recent paper of Corso and Shmerkin \cite{CoShm24}. In order to state their result, we need to
recall the notion of $L^q$ dimension for a probability measure $\nu$ in $\R^d$. It is given by
$$
\dim_q(\nu) = \frac{\tau_q(\nu)}{q-1},\ \ q>1,\ \ \mbox{where}\ \ 
\tau_q(\nu) = \liminf_{n\to\infty} \frac{-\log\sum_{Q\in D_n}\nu(Q)^q}{n},
$$
and $D_n$ is the partition of $\R^d$ into $2^{-n}$-mesh cubes. The ``self-similar'' $L^q$ dimension of a homogeneous self-similar
measure $\mu$, corresponding to a weighted ISF with a contraction ratio $\thet^{-1}$ and probability vector $\pb = (p_j)$, provides a
natural upper bound for the $L^q$-dimension of $\mu$:
\beq \label{eq:dimsq}
\dim_q(\mu) \le \dim_{\rm s}(\mu, q):= \frac{-\log \sum_j p_j^q}{(q-1)\log\thet}.
\eeq
Now we can state \cite[Theorem 1.11]{CoShm24}, which is slightly reformulated for our purposes.

\begin{theorem}[E. Corso and P. Shmerkin, 2024 preprint] \label{th:CoSh}
Let $d=2s+r$, with $r\in \{0,1\}$. Consider a homogeneous self-similar IFS in $\C^s\oplus \R^r \cong \R^d$ of the following form. 
Let $\thet>1$ and $\alpha_k\in [0,1)$, $1\le k \le s$, be
such that $\alpha_k\notin \Q$ and $\alpha_i - \alpha_k\notin \Q$ for $i\ne k$ (in particular, $\alpha_i \ne \alpha_k$). The IFS is
$\Phi = \{f_j\}_{j=0}^m$ where $f_j(\bx) = \thet^{-1} \Ok(\bx) + a_j$, with 
$$
\Ok = {\rm Diag}[e^{-2\pi i \alpha_1},\ldots, e^{-2\pi i \alpha_s}]\ \ \mbox{and}\ \
a_j = (a_j^{(1)}, \ldots, a_j^{(s)})\in \C^s,\ \ \mbox{if}\ d=2s,
$$
or
$$
\Ok = {\rm Diag}[e^{-2\pi i \alpha_1},\ldots, e^{-2\pi i \alpha_s},\pm 1]\ \ \mbox{and}\ \
a_j = (a_j^{(1)}, \ldots, a_j^{(s)}, a_j^{(d)})\in \C^s\oplus \R,\ \ \mbox{if}\ d=2s+1.
$$

{\bf  Main Assumption:} All the  coordinate IFS of the form 
$\{f_j^{(\ell)}\}_{j=0}^m$, where
$$
f_j^{(\ell)}(z) = \thet^{-1} e^{-2\pi i \alpha_\ell} z + a_j^{(\ell)},\ \ \ell=1,\ldots,s,\ \ \mbox{on}\ \ \C,
$$
and
$$
f_j^{(d)}(x) = \pm \thet^{-1} x + a_j^{(d)},\ \ \mbox{if}\ d=2s+1,\ \ \mbox{on}\ \ \R,
$$
fulfil the exponential separation condition. 

Then, for any probability vector $\pb>0$, the associated
self-similar measure satisfies
$$
\dim_q(\mu) = \min\{d,\dim_{\rm s}(\mu,q)\}.
$$
\end{theorem}

\begin{remark}
{\em 
(a) The results of Corso and Shmerkin \cite{CoShm24} are much more general; they work in the framework of ``pleasant models,''
see \cite[Section 1]{CoShm24}.

(b)  The {\em Main Assumption}
 in the last theorem is a special case of the ``projected exponential separation'' condition, see \cite[Definition 1.10]{CoShm24}.

(c)  Corso and Shmerkin \cite{CoShm24} obtain  analogous results on $L^q$ dimensions of projected measures to {\em arbitrary subspaces}.
}
\end{remark}

\begin{corollary}\label{cor:ac1}
\begin{enumerate}
\item[{\bf (i)}] Let $d\ge 3$ be arbitrary, and let $\Dk =\{0, a_1,\ldots, a_m\}$ be an arbitrary spanning digit set in $\R^d$. Then for $\Lk^1\times m_d$-a.e.\
contracting similitude $A = \thet^{-1} \Ok$ in $\R^d$, for every probability vector $\pb$ and $q>1$, such that 
\beq \label{eq:lq-cond}
\dim_s(\mu)=\frac{-\log \sum_j p_j^q}{(q-1)\log\thet}>d,
\eeq
the self-similar measure $\mu(\thet^{-1}\Ok, \Dk,\pb)$ is absolutely continuous with respect to $\Lk^d$ with a density (the Radon-Nikodym 
derivative) in $L^q(\R^d)$.
\item[{\bf (ii)}] 
For a.e.\ self-similar homogeneous IFS in $\R^d \cong \C^s$, $d=2s$, with a contraction ratio $\thet^{-1}\in (0,1)$,
which is affinely irreducible, for every $\pb>0$ and $q>1$, satisfying \eqref{eq:lq-cond},
the corresponding self-similar measure is absolutely continuous with respect to $\Lk^d$ with a density in $L^q(\R^d)$.
\end{enumerate}
\end{corollary}

Let $d=2s$, and consider a self-similar IFS in $\C^s\cong \R^d$, with $A = \thet^{-1}\Ok$ a complex diagonal $s\times s$ matrix and
$\Dk =\{0, a_1,\ldots, a_m\}$ an arbitrary set of digits, with $m\ge 1$, such that for each $\ell\in [s]$ there exists $a_j^{(\ell)}\ne 0$.Then for
(Lebesgue)-a.e. such $A$, for every probability vector $\pb$ and $q>1$, satisfying \eqref{eq:lq-cond},
the self-similar measure $\mu(\thet^{-1}\Ok, \Dk,\pb)$ is absolutely continuous with respect to $\Lk^d$ with a density in $L^q(\R^d)$.

\begin{proof}[Proof (derivation)]
(i)
By analogy with \eqref{eq:split1} we have, for $k\ge 2$:
\beq \label{eq:split2}
\mu=\mu(\thet^{-1}\Ok, \Dk,\pb) = \mu(\thet^{-k}\Ok^k, \Dk,\pb) * \mu(\thet^{-k}\Ok^k, \wt\Dk_k,\wt\pb_k) =: \mu_k * \wt\mu_k,
\eeq
where $\mu_k=\mu(\thet^{-k}\Ok^k, \Dk,\pb)$ corresponds to keeping  every $k$-th term in the convolution product formula \eqref{eq-conv} and skipping the rest, whereas other measure, $\wt\mu_k$, is the convolution product with every $k$-th term skipped.
It is easy to see that  $\wt\mu_k=\mu(\thet^{-k}\Ok^k, \wt\Dk_k,\wt\pb_k)$
is indeed a self-similar measure with some other digit set $\wt\Dk_k$ (depending on $\thet^{-1}\Ok$) and a probability 
vector $\wt\pb_k$.  Fix $q>1$ and a probability vector $\pb$. It suffices to prove the claim for a.e.\ $\thet>1$ satisfying
$$
\dim_{\rm s}(\wt\mu_k, q)=\frac{-(k-1)\log \sum_j p_j^q}{k(q-1)\log\thet}>d,
$$
for every $k\ge 2$, so we fix $k$ for the rest of the proof. The condition on $\Ok$ in Theorem~\ref{th:CoSh}, that is, on its eigenvalues, is clearly a full 
measure condition. 
Then by Theorem~\ref{th:new1}, for all but a zero Hausdorff dimension set of exceptions $\thet$, the
measure $\mu_k = \mu(\thet^{-k}\Ok^k, \Dk,\pb)$ has power Fourier decay.  For the measure $\wt\mu_k$ we
would like to apply Theorem~\ref{th:CoSh}, then by \cite[Theorem 4.4]{ShSol16} the claim on absolute continuity with a density in $L^q$ follows.
It remains to ensure exponential separation (ES) for the coordinate IFS's. A standard by now argument (see \cite[Theorem 5.11]{Hochman} and
\cite[Prop.\,3.2]{ShSol16b}) yields that for every coordinate there is a zero packing dimension set of exceptional contraction coefficients
 ($\thet^{-1}$ in the case of
the real eigenvalue and $\thet^{-1} e^{-2\pi i \alpha_\ell}$ in the case of a complex eigenvalue), on the complement of which the ES holds. 
Combining everything, we obtain the desired zero-Lebesgue measure exceptional set.

\smallskip

(ii) In view of the remarks above, we can assume without loss of generality that we have
a self-similar IFS in $\C^s\cong \R^d$, with $A = \thet^{-1}\Ok$ a complex diagonal $s\times s$ matrix and
$\Dk =\{0, a_1,\ldots, a_m\}$ is an arbitrary set of digits, such that for each $\ell\in [s]$ there exists $a_j^{(\ell)}\ne 0$.
Next we use the same scheme, representing the measure as a convolution \eqref{eq:split2}. For every $\ell\in [s]$, fixing all the eigenvalues of $A$, but 
$\theta_\ell$, compatible with the requirements of Theorem~\ref{th:CoSh}, we obtain a zero-dimensional set of exceptions for $\theta_\ell$, such that
in its complement both the exponential separation condition in the $\ell$-th coordinate holds, and uniform Fourier power decay holds in the directions having
the $\ell$'s coordinate dominant. This yields a zero Lebesgue measure set of exceptions in the space of the corresponding similitudes. Their union over 
$\ell$ is the desired exceptional set.
\end{proof}


\section{Preliminary steps of the proof of Fourier decay} \label{sec:prelim} 

By the definition of the self-similar measure,
$$
\muhat(\xi) = \sum_{j=0}^m p_j \int e^{-2\pi i \langle \xi, Ax + a_j\rangle}\, d\mu(x) = \Bigl(\sum_{j=0}^m p_j  e^{ -2\pi i \langle \xi,  a_j\rangle}\Bigr) \muhat(A^t\xi),
$$
where $A^t$ is the matrix transpose of $A$. Iterating we obtain
\begin{equation} \label{eq1}
\muhat(\xi) = \prod_{n=0}^\infty\left(\sum_{j=0}^m p_j  e^{-2\pi i \langle (A^t)^n\xi,  a_j\rangle}\right) = \prod_{n=0}^\infty\left(\sum_{j=0}^m p_j  e^{-2\pi i \langle \xi,  A^n a_j\rangle}\right),
\end{equation}
where the infinite product converges, since $\|A^n\| \to 0$ exponentially fast.
Let $\mu = \mu(\thet^{-1}\Ok, \Dk, \pb)$.
For $\xi\in \R^d$, with $\|\xi\|_\infty \ge 1$, consider $\eta(\xi) = (A^t)^{N(\xi)} \xi$, where $N(\xi)\ge 0$ is maximal, such that $\|\eta(\xi)\|_\infty \ge 1$. 
Then $\|\eta(\xi)\|_\infty \in [1, \thet)$ and \eqref{eq1} implies for $\eta = \eta(\xi,A)$:
\begin{equation} \label{eq2}
\muhat(\xi) = \muhat(\eta)\cdot \prod_{n=1}^{N(\xi)} \left(\sum_{j=0}^m p_j  e^{-2\pi i \langle \eta,\,  A^{-n} a_j\rangle}\right)
\end{equation}

The following is an elementary inequality; for the proof, see e.g., \cite[Lemma 2.1]{Sol22}.

\begin{lemma} \label{lem:elem} 
Let $\bp = (p_0,\ldots,p_m)>0$ be a probability vector and $\alpha_0=0,\ \alpha_j \in \R$, $j=1,\ldots,m$. Denote $\eps = \min_j p_j$ and write 
${\|x\|}_{\scriptscriptstyle \R/\Z}=\dist(x,\Z)$. Then
for any $k\le m$,
\beq \label{elem}
\left|\sum_{j=0}^m p_j e^{-2\pi i \alpha_j}\right| \le 1 - 2\pi\eps {\|\alpha_k\|}_{\scr}^2.
\eeq
\end{lemma}

Thus, using that  $|\muhat(\eta)|\le 1$, we obtain
\beq \label{imp1}
|\muhat(\xi)| \le  \prod_{n=1}^{N(\xi)} \left(1-2\pi \eps \max_{1\le j \le m} 
\Bigl\| \langle \eta, A^{-n} a_j\rangle \Bigr\|_{\scr}^2  \right) =: \Psi(A,\Dk,\eps,\xi). 
\eeq

We summarize the above discussion in the following

\begin{lemma}[Reduction] \label{lem:redu}
In order to show that $\mu(A,\Dk,\pb)\in \Dsc_d$, with $\pb>0$, it suffices to verify that there exists $\gam>0$, such that 
$$
\Psi(A, \Dk',\eps,\xi) = O(|\xi|^{-\gam}),\ \ |\xi|\to \infty,
$$
for any $\Dk' \subseteq \Dk$ with $0 \in \Dk'$ and $\eps=\min_j p_j$.
\end{lemma}


\section{Proof of Theorem~\ref{th:new1}}

Recall that now $\Dk = \{0,a_1,\ldots,a_m\}$ is a spanning set for $\R^d$. By Lemma~\ref{lem:redu}, we can assume that $m=d$
and  obtain from \eqref{imp1}:
\beq \label{imp11}
|\muhat(\xi)| \le 
=\prod_{n=1}^{N(\xi)} \left(1-2\pi \eps \max_{1\le j \le d} 
\Bigl\| \langle  a_j, \thet^n(\Ok^t)^{-n}\eta\rangle \Bigr\|_{\scr}^2  \right).
\eeq

As already mentioned above, the proof uses a variant of the Erd\H{o}s-Kahane method and has some common features with the proofs in
\cite{ShSol16,ShSol16b} and other applications of this technique.

We start by defining 
the exceptional set in Theorem~\ref{th:new1}. For $\delta, \rho >0$ and $1< B_1 < B_2$ we first define the ``bad set'' at scale $N$:
\beq \label{exc1}
E_N(\delta, \rho, B_1,B_2) = \left\{\thet\in [B_1, B_2]:\!\!\! \max_{\eta:\ 1 \le \|\eta\|_\infty \le B_2} \frac{1}{N} 
\Bigl|\Bigl\{n \in [N]:\!\! \max_{1\le j \le d} \Bigl\| \langle  a_j, \thet^n(\Ok^t)^{-n}\eta\rangle \Bigr\|_{\scr} \!\!\!< \rho\Bigr\}\Bigr| > 1-\delta \right\}.
\eeq
The exceptional set is given by
$$
\Ek(\delta,\rho,B_1,B_2)= \bigcap_{N_0=1}^\infty \bigcup_{N=N_0}^\infty E_{N}(\delta,\rho,B_1,B_2).
$$
Theorem~\ref{th:new1} is immediate from the next two propositions. We assume that the orthogonal matrix $\Ok$ and a spanning digit set
$\Dk = \{0,a_1,\ldots,a_d\}$ are fixed. All the constants below depend on them.

\begin{prop} \label{prop-excep00} Let $1  < B_1 < B_2$.
For any $\delta>0$, and $\rho\in (0,\half)$, we have $\mu(\thet^{-1}\Ok,\Dk,\pb)\in \Dsc_d(\gam)$ for some $\gam>0$,
whenever $\thet\not\in \Ek(\delta,\rho,B_1,B_2)$.
The power $\gam$ depends only on $\delta,\rho, B_2$, and $\eps = \min_j p_j$.
\end{prop}

\begin{sloppypar}
\begin{prop} \label{prop-dim00} For any  $\beta>0$ and $1<B_1<B_2$ there 
exist  $\delta>0$ and $\rho \in (0,\half)$ such that $\Dh(\Ek(\delta,\rho,B_1,B_2))\le\beta$.
\end{prop}
\end{sloppypar}

The proof of Proposition~\ref{prop-dim00} is based on the following

\begin{prop} \label{propa-EK00}
For any $1< B_1< B_2$ there exist  positive constants $\rho=\rho(B_1,B_2)$, $C_0 = C_0(B_1,B_2)$, and $n_1 = n_1(B_1,B_2)\in \N$,
such that, for any $N\ge n_1$ and $\delta\in (0,\half)$, the 
set $E_{N}(\delta,\rho,B_1,B_2)$ can be covered
by $\exp\bigl(C_0\cdot\delta \log(1/\delta)N\bigr)$ disks of diameter $B_1^{-N}$.
\end{prop}

We first complete the proof of Proposition~\ref{prop-dim00}, assuming the last proposition. The symbol $\Hk^\beta$ denotes the 
$\beta$-dimensional Hausdorff measure.
By Proposition~\ref{propa-EK00}, for $N_0\ge n_1$,
$$
\Hk^\beta\left(\bigcup_{N=N_0}^\infty E_{N}(\delta,\rho,B_1,B_2)\right) \le 
\sum_{N=N_0}^\infty \exp(C_0\cdot\delta \log(1/\delta)N)\cdot B_1^{-\beta N} \to 0, \ \ N_0\to \infty,
$$
provided $\delta>0$ is so small that $\beta\log B_1 >C_0\cdot \delta \log(1/\delta)$. Thus $\Dh(\Ek(\delta,\rho,B_1,B_2))\le \beta$. \qed



It remains to prove Propositions \ref{prop-excep00} and \ref{propa-EK00}.

\begin{proof}[Proof of Proposition~\ref{prop-excep00}]
Suppose that $\thet\notin  \Ek(\delta,\rho,B_1,B_2)$. This implies that there exists $N_0\ge 1$
such that $\thet \notin E_N(\delta,\rho, B_1,B_2)$ for all $N\ge N_0$.

Let $\xi\in \R^d$ be such that $\|\xi\|_\infty \ge \vartheta^{N_0}$. As above, let $N = N(\xi)\ge 1$ be maximal,
such that $\|\eta\|_\infty\ge 1$ for $\eta = (A^t)^N\xi$.
Then $N \ge N_0$ and $\|\eta\|_\infty\le \vartheta\le B_2$.
 From the fact that $\thet \notin E_N(\delta,\rho, B_1,B_2)$ it follows that 
$$
 \max_{1\le j\le d}\,N^{-1}\cdot\Bigl|\Bigl\{n\in [N]:\ 
 \Bigl\| \langle  a_j, \thet^n(\Ok^t)^{-n}\eta\rangle \Bigr\|_{\scr} < \rho\Bigr\}\Bigr| 
\le 1-\delta.
$$
Then by \eqref{imp11},
$$
\bigl|\muhat(\thet^{-1}\Ok, \Dk, \pb)(\xi)\bigr| \le (1 - 2\pi\eps \rho^2)^{\lfloor\delta N\rfloor} \le (1 - 2\pi\eps \rho^2)^{\delta N-1}.
$$
By the definition of $N=N(\xi)$ we have 
$
\|\xi\|_\infty \le \vartheta^{N+1} \le B_2^{N+1};
$
 thus it follows that
$$
\bigl|\muhat(\thet^{-1}\Ok, \Dk, \pb)(\xi)\bigr| = 
O_{\mu}(1)\cdot {\|\xi\|}_{_{\scriptstyle{\infty}}}^{-\gam},
$$
for $\gam = -\delta\log (1 - 2\pi\eps \rho^2)/\log B_2$, and the proof is complete.
\end{proof}

\medskip

\begin{proof}[Proof of Proposition~\ref{propa-EK00}]
Let us write
\beq \label{def-Kn0}
\langle  a_j, \thet^n(\Ok^t)^{-n}\eta\rangle =  K_n^{(j)}+\eps^{(j)}_n,\ \ n\ge 0,\ \ 1\le j\le d.
\eeq
where $K^{(j)}_n\in \Z$ is the nearest integer to the expression in the left-hand side, so that, by definition,
$$
|\eps^{(j)}_n| = \Bigl\| \langle  a_j, \thet^n(\Ok^t)^{-n}\eta\rangle \Bigr\|_{\scr} \le 1/2.
$$ 
Note that \eqref{def-Kn0} may be viewed as a non-singular system of linear equations, by the assumption that $\{a_j\}_{j=1}^d$ forms a basis of $\R^d$. 
Denote by $T_\Dk$ the matrix whose rows are $a_j^t$. Then we obtain from \eqref{def-Kn0}:
\beq \label{eq:new0}
\vec{K}_n + \vec{\eps}_n = T_D \Ok^t)^{-n}\thet^n \eta \implies  \thet^n \eta = (\Ok^t)^{n} T_\Dk^{-1} (\vec{K}_n + \vec{\eps}_n),
\eeq
where $\vec{K}_n$ and $\vec{\eps}_n$ are vectors with components $K_n^{(j)}$ and $\eps^{(j)}_n$ respectively. Observe that
$$
{\|\vec\eps_n\|}_\infty = \max_{1\le j\le d} \Bigl\| \langle  a_j, \thet^n(\Ok^t)^{-n}\eta\rangle \Bigr\|_{\scr}.
$$
Denote
$$
\vec{L}_n = (\Ok^t)^{n} T_\Dk^{-1} \vec{K}_n.
$$
The idea for what comes next is to show that  $\thet \approx \|\vec{L}_{n+1}\|/\|\vec{L}_n\|$, and hence
$$
\vec{K}_{n+1}\approx \frac{\|\vec{L}_{n+1}\|}{\|\vec{L}_n\|}\cdot T_\Dk\,(\Ok^t)^{-1} \,T_\Dk^{-1} \vec{K}_n
$$
is a good approximation when $\max\{\|\vec{\eps}_n\|, \|\vec{\eps}_{n+1}\|\}$ is small. Below we are using the $\ell^\infty$ norm, unless stated
otherwise. Recall that $\Ok$ and $\Dk$ are fixed.

\begin{lemma} \label{lem:new1}
There exists $n_1= n_1(B_1,B_2) \in \N$, such that
\beq \label{eq:new1}
\left|\thet - \frac{\|\vec{L}_{n+1}\|}{\|\vec{L}_n\|}\right| \le C_1\cdot \thet^{-n} \max\{\|\vec{\eps}_n\|,\|\vec{\eps}_{n+1}\|\}
\le C_1\cdot B_1^{-n} \max\{\|\vec{\eps}_n\|,\|\vec{\eps}_{n+1}\|)\},\ \ n\ge n_1,
\eeq
and
\beq \label{eq:new2}
\left\|\vec{K}_{n+1} -  \frac{\|\vec{L}_{n+1}\|}{\|\vec{L}_n\|}\cdot T_\Dk\,(\Ok^t)^{-1} \,T_\Dk^{-1} \vec{K}_n\right\| 
< C_2\cdot \max\{\|\vec{\eps}_n\|,\|\vec{\eps}_{n+1}\|)\},\ \ n\ge n_1,
\eeq
where $C_1=C_1(B_1, B_2)$ and $C_2 = C_2(B_1, B_2)$.
\end{lemma}

\begin{proof}
Note that 
\beq \label{eqnu1}
{\|\vec{L}_n - \thet^n\eta\|}_\infty = {\|(\Ok^t)^{n} \,T_\Dk^{-1} \vec{\eps}_n\|}_\infty\le 
\sqrt{d}\cdot {\|T_\Dk^{-1}\|}_\infty 
\cdot{\|\vec{\eps}_n\|}_\infty,
\eeq
since 
$$\|(\Ok^t)^n\|_2 =1 \implies \|(\Ok^t)^n\|_\infty \le \sqrt{d}.$$
It follows that
$$
\thet^n - \textstyle{\frac{\sqrt{d}}2} {\|T_\Dk^{-1}\|}_\infty\le {\|\vec{L}_n\|}_\infty \le B_2\thet^n + \textstyle{\frac{\sqrt{d}}2} {\|T_\Dk^{-1}\|}_\infty,
$$
hence
\beq \label{eqnu2}
{\|\vec{L}_n\|}_\infty \in \left[\thet^n/2, (B_2+1/2)\cdot\thet^n\right]\ \
\eeq
for $n$ sufficiently large, depending on $B_1$, and $B_2$. (Recall that dependence on $\Dk$ is assumed by default.)
Now, writing:
$$
\thet - \frac{\|\vec{L}_{n+1}\|}{\|\vec{L}_n\|} = 
\frac{\bigl(\|\vec{L}_n\|  - \|\thet^n\eta\|\bigr) \cdot \|\thet^{n+1}\eta\|}{\|\thet^n \eta\|\cdot \|\vec{L}_n\|}- \frac{\|\vec{L}_{n+1}\|  - \|\thet^{n+1}\eta\|}{\|\vec{L}_n\|}
$$
and using \eqref{eqnu1}, \eqref{eqnu2}, we obtain \eqref{eq:new1}, with
$$
C_1 = 2\sqrt{d}\cdot {\|T_\Dk^{-1}\|}_\infty\cdot (B_2+1).
$$

Next we show \eqref{eq:new2}.
We have,
in view of \eqref{eq:new0}:
$$
\vec{K}_{n+1} + \vec{\eps}_{n+1} = \thet \cdot T_\Dk \, (\Ok^t)^{-1}\, T_\Dk^{-1} (\vec{K}_n + \vec{\eps}_n),
$$
hence
$$
\bigl\|\vec{K}_{n+1} - \thet\cdot T_\Dk \, (\Ok^t)^{-1}\, T_\Dk^{-1} \vec{K}_n\bigr\|\le \|\vec{\eps}_{n+1}\| + B_2\sqrt{d}\cdot \|T_\Dk\|\cdot \|T_\Dk^{-1}\|\cdot
\|\vec{\eps}_{n}\|.
$$
It remains to combine this with \eqref{eq:new1}, noting that 
\beq \label{eqnu3}
\|\vec{K}_n\| \le \sqrt{d}\cdot B_2 \|T_\Dk\|\cdot \thet^n + 1/2 \le \sqrt{d}\cdot (B_2+1) \|T_\Dk\|\cdot \thet^n,
\eeq
for $n$ sufficiently large. We obtain
\begin{eqnarray*}
\left\|\vec{K}_{n+1} -  \frac{\|\vec{L}_{n+1}\|}{\|\vec{L}_n\|}\cdot T_\Dk\,\Ok^t \,T_\Dk^{-1} \vec{K}_n\right\| 
& \le & ( 1 +  B_2\sqrt{d}\cdot \|T_\Dk\|\cdot \|T_\Dk^{-1}\|) \cdot \max\{\|\vec{\eps}_n\|,\|\vec{\eps}_{n+1}\|)\} \\
& + & C_1\thet^{-n} \cdot \max\{\|\vec{\eps}_n\|,\|\vec{\eps}_{n+1}\|)\}\cdot \sqrt{d}\cdot \|T_\Dk\|\cdot  \|T_\Dk^{-1}\|\cdot \|\vec{K}_n\| \\[1.1ex]
& < & C_2\cdot \max\{\|\vec{\eps}_n\|,\|\vec{\eps}_{n+1}\|)\},
\end{eqnarray*}
in view of \eqref{eqnu3}, for some $C_2 = C_2(B_1,B_2)$ and $n\ge n_1(B_1, B_2)$, as desired.
\end{proof}

\medskip

Now we can conclude the proof by a standard application of the Erd\H{o}s-Kahane scheme. We follow \cite{ShSol16} closely. Let
\beq \label{def-rho0}
\rho:= (2C_2)^{-1}\ \ \ \mbox{and}\ \ \ M:= (2C_2 + 1)d.
\eeq
The inequality \eqref{eq:new2} in Lemma~\ref{lem:new1}  implies the following, since $\vec K_n$ are integer vectors.

\begin{lemma} \label{lem:erd20}
Suppose that the orthogonal matrix $\Ok$ and the spanning digit set $\Dk = \{0,a_1,\ldots,a_d\}$ are given.
Consider an arbitrary $\thet\in [B_1,B_2]$ and $\eta \in \R^d$, 
with $1\le \|\eta\|_\infty\le  B_2$. Define the vectors $\vec K_n, \vec \eps_n$ by \eqref{def-Kn0} and \eqref{eq:new0}.
Then the following hold, independent of $\thet$ and $\eta$:
\begin{enumerate}
\item[{\bf (i)}] for any $n\ge n_1(B_1, B_2)$, such that $\max\{\|\vec\eps_n\|, \|\vec\eps_{n+1}\|\}<\rho$,
 the vector $\vec K_{n+1}$ is uniquely determined by $\vec K_n$.
\item[{\bf (ii)}] for any $n\ge n_1(B_1, B_2)$ there are at most $M$ choices of $\vec K_{n+1}$ given $\vec K_n$.
\end{enumerate}
\end{lemma}

Now we can finish the proof of Proposition~\ref{propa-EK0}. Suppose that $N\ge n_1(B_1,B_2)$. 
Fix $\thet\in E_N(\delta,\rho,B_1,B_2)$, with $\rho$ from \eqref{def-rho0} and $\delta\in (0,\half)$ arbitrary. By definition, there exists
$\eta \in \R^d$, with
 $1\le \|\eta\|_\infty\le B_2$, satisfying
$$
\Bigl|\Bigl\{n\in [N]:\ \max_{1\le j \le d} \Bigl\| \langle  a_j, \thet^n(\Ok^t)^{-n}\eta\rangle \Bigr\|_{\scr}< \rho\Bigr\}\Bigr| \ge (1-\delta)N.
$$
Define the sequence of vectors $\vec K_n$ and $\vec \eps_n$ associated with $\thet$ and $\eta$, so that
$$
{\|\vec\eps_n\|}_\infty = \max_{1\le j\le d} \Bigl\| \langle  a_j, \thet^n(\Ok^t)^{-n}\eta\rangle \Bigr\|_{\scr}.
$$
By assumption, there are $O_{B_1,B_2}(1)$ choices for the initial part of the sequence $\vec K_1,\ldots, \vec K_{n_1}$. 
The set $J := \{n\in [N]: \|\vec\eps_n\|\ge \rho\}$ has cardinality at most 
$\lfloor \delta N \rfloor$, by construction. In view of Lemma~\ref{lem:erd20}, given $J$, there are at most $O_{B_1,B_2}(M^{3\delta N})$ choices for the 
sequence $\vec K_1,\ldots, \vec K_{N+1}$. By \eqref{eq:new1}, the minimal number of disks of radius $C_1B_1^{-n}$ needed to cover 
the set $E_N(\delta,\rho,B_1,B_2)$, is at most
$$
O_{B_1,B_2}(M^{3\delta N})\cdot {N\choose \lfloor \delta N \rfloor} = \exp(C_0\cdot\delta \log(1/\delta)N),
$$
for some $C_0 = C_0(B_1,B_2)$,
as desired.
\end{proof}


\section{Proof of Theorem~\ref{thm2}}

\subsection{Reduction}  \label{sec:reduc}
Recall that we consider a self-similar IFS $f_j(x) =Ax + a_j$ in $\R^d=\R^{2s} \cong \C^s$, with $A = \thet^{-1}\Ok$, where $\Ok\in O(d,\R)$ is given
by 
$$
\Ok = {\rm Diag}[e^{-2\pi i \alpha_1},\ldots, e^{-2\pi i \alpha_s}],\ \ \alpha_j \in (0,1/2),
$$
and the digit set is $\Dk=\{0,a_1,\ldots,a_m\}$, with $a_j = (a_j^{(1)}, \ldots, a_j^{(s)})\in \C^s$. It will be convenient to view
the IFS and its Fourier transform in the complex space, so that we identify $\xi\in \R^{2s}$ with $\zeta=(\zeta_1,\ldots,\zeta_s)\in \C^s$,
such that $\zeta_\ell = \xi_{2\ell-1} + i\xi_{2\ell},\ \ell=1,\ldots,s$.

The preliminary steps are essentially the same as in Section~\ref{sec:prelim}, but in complex notation.
We have from \eqref{eq1}:
\beq \label{eq11}
\muhat(\xi) = \muhat(\zeta)=\prod_{n=0}^\infty\left(\sum_{j=0}^m p_j e^{-2\pi i A^n a_j\dprod\zeta}\right),
\eeq
where the dot product in $\C^s$ is given by $z_1\dprod z_2 = \sum_{\ell=1}^s \re(z^{(\ell)}_1\ov z^{(\ell)}_2)$. Note that
$$
A^na_j \dprod \zeta = a_j\dprod (A^*)^n\zeta,\ \ \mbox{where}\ \ A^* = \thet^{-1} {\rm Diag}[e^{2\pi i \alpha_1},\ldots, e^{2\pi i \alpha_s}].
$$
We next record a minor variant of \eqref{eq2}, with the only difference being that $\|\zeta\|_\infty$ is, in general, not equal to $\|\xi\|_\infty$.
For $\bzeta\in\C^s$, with $\|\bzeta\|_\infty \ge 1$, let $\eta=\eta(\bzeta,A) = (A^*)^{N(\bzeta)} \bzeta$, where 
$N(\bzeta)\in \N$ is maximal, such that $\|\eta\|_\infty \ge 1$. 
Then $\|\eta\|_\infty \in [1, \thet)$, and \eqref{eq11} implies for $\eta = (\eta^{(\ell)})_{\ell\le s}\in \C^s$:
\begin{equation} \label{eq20}
\muhat(\bzeta) = \muhat(\eta)\cdot \prod_{n=1}^{N(\bzeta)} \left(\sum_{j=1}^m p_j  e^{-2\pi i A^{-n} a_j \dprod \eta}\right).
\end{equation}
We have
$$
A^{-n} a_j \dprod \eta = \thet^n \sum_{\ell=1}^s \re\left[e^{2\pi i n\alpha_\ell} a_j^{(\ell)} \ov\eta^{(\ell)}\right],
$$
hence \eqref{eq20} and Lemma~\ref{lem:elem} imply (using that $|\muhat(\eta)|\le 1$) the inequality
\beq \label{imp12}
|\muhat(\bzeta)| \le  \prod_{n=1}^{N(\bzeta)} \left(1-2\pi \eps \max_{j\le m} 
\Bigl\| \thet^n \sum_{\ell=1}^s \re (e^{2\pi i n\alpha_\ell} a_j^{(\ell)} \ov\eta^{(\ell)})\Bigr\|^2  \right),
\eeq
where $\eps = \min_j p_j>0$. Recall that for every $\ell\le s$ 
there exists $j$ such that $a_j^{(\ell)}\ne 0$ (this is the condition for the IFS to be affinely irreducible).
By a linear change of variable (applying a conjugacy defined by a complex diagonal matrix), 
we can assume, without loss of generality, that for every $\ell$
there exists $j$ such that $a_j^{(\ell)} = 2$. Thus 
we obtain from \eqref{imp12}:
\beq \label{imp2}
|\muhat(\bzeta)| \le  \prod_{n=0}^{N(\bzeta)}\left(1-2\pi \eps 
\Bigl\| \sum_{\ell=1}^s 2\thet^n \,\re(e^{2\pi i n\alpha_\ell} \,
\ov\eta^{(\ell)})\Bigr\|^2  \right) =: \Psi(\thet, \balpha,\eps,\zeta),
\eeq
where $\balpha = (\alpha_1,\ldots,\alpha_s)$.

In what follows, it will be convenient to make a further notational change; essentially passing from
$\C^s$ to $\C^{2s}$.  We use the notation $[s]=\{1,\ldots,s\}$.
Let $\thb = (\theta_1,\ldots,\theta_{2s})$ be the full list of eigenvalues of $A^{-1}$, so that
$$
\theta_{2j} = \ov \theta_{2j-1} = \thet\cdot e^{2\pi i \alpha_j},\ \ j\in [s].
$$
Further, for $\zeta\in \C^{s}$ let $\bz(\zeta) \in \C^{2s}$ be given by
$$
\bz=\bz(\zeta)=(z_1,\ldots,z_{2s}), \ \ \mbox{where}\ \ z_{2j-1} = \zeta_j,\ z_{2j} = \ov\zeta_j,\ j\in [s],
$$
and $\tau(\bz)=\tau(\eta(\zeta))\in \C^{2s}$ is
$$
\tau(\bz) = (\tau_1,\ldots,\tau_{2s}),\ \ \mbox{with}\ \ \tau_{2j-1} = \ov\eta^{(j)},\ \tau_{2j} = \eta^{(j)},\ j\in [s]. 
$$
Note that
\beq\label{eq-tau}
\tau_j = \theta_j^{-N(\bz)}z_j,\  \ j\in [2s],\ \mbox{where}\ \ N(\bz)\ \ \mbox{is such that}\ \ \|\tau\|_\infty = \thet^{-N(\bz)}\|\bz\|_\infty \in [1,\thet).
\eeq
Equation \eqref{imp2} implies
\beq \label{imp22}
\Psi(\thet, \balpha,\eps,\zeta) = \prod_{n=0}^{N(\bz)} \left(1 - 2\pi \eps\Bigl\|\sum_{j=1}^{2s} \theta_j^n \tau_j\Bigr\|^2\right) =:
\Psi(\thb,\eps,\bz).
\eeq

We obtain

\begin{lemma} \label{lem:reduc}
In order to show that $\mu(\thet^{-1}\Ok,\Dk,\pb)\in \Dsc_d$, with $\pb>0$, it suffices to verify that there exists $\gam>0$ such that
$$
\Psi(\thb,\eps,\bz)= O(|\bz|^{-\gam}),\ \ |\bz|\to \infty,
$$
with  $\eps=\min_j p_j$.
\end{lemma}


\subsection{Beginning of the proof of Theorem~\ref{thm2}}
This theorem follows from the next claim, which is our main technical result. Given $s\ge 2$, $\vartheta>1$, $b_1, b_2>0$, let
\begin{eqnarray*}
H = H(s,\vartheta,b_1,b_2) & := & \Bigl\{\thb = (\theta_1,\ldots,\theta_{2s})\in \C^{2s}: \  \ \mbox{for all}\ j\in [s]\ \ \mbox{holds}\ \ |\theta_{2j-1}|=\thet,\\
& &  \im(\theta_{2j-1})\ge b_1,\ \ \theta_{2j} = \ov\theta_{2j-1};\ \ |\theta_i - \theta_j| \ge b_2,\  i\ne j\Bigr\}.
\end{eqnarray*}
Fix $H$ for the rest of this proof. Recall that $\T^+:= \{e^{2\pi i \alpha}:\ \alpha\in (0,\half)\}$.

\begin{theorem}\label{th:tech}
Let $H = H(s,\vartheta,b_1,b_2)$ be as above, and let $\ell \in [s]$. 
Fix $\theta_j$ for $j\in [2s]\setminus \{2\ell-1,\ell\}$, compatible with $\thb\in H$.
Further, fix $\beta, \eps>0$. Then there exist $\gam_\ell = \gam_\ell(H,\beta, \eps)$ and $\Ek_\ell\subset \thet\cdot\T^+\subset \C$ 
satisfying the following:
\ben
\item[{\rm (i)}] $\Dh(\Ek_\ell) \le \beta$;
\item[{\rm (ii)}] for all $\thb\in H$ such that $\theta_{j}$,\ $j\in [2s]\setminus \{2\ell-1,\ell\}$, are those fixed,
$\theta_{2\ell-1} \in\thet \cdot \T^+\setminus\Ek_\ell$ is arbitrary, $\theta_{2\ell} = \ov\theta_{2\ell-1}$, for any probability vector $\pb$, with
$\min_i p_i \ge \eps$, holds
\beq \label{new-decay}
 \Psi(\thb,\eps,\bz) \le O_{H}(1)\cdot \|\bz\|_{\infty}^{-\gam_\ell}\ \ \mbox{for all}\ \bz\in \C^{2s}\ \mbox{such that} \
\ \ |z_\ell| = \|\bz\|_\infty \ge O_{\thb}(1).
\eeq

\een
\end{theorem}

\begin{proof}[Derivation of Theorem~\ref{thm2}]
Let $d=2s$. The set of almost all 
contracting similitudes, up to conjugacy, can be expressed as a union of the sets $H(s,\thet, b_1, b_2)$, with $\thet>1$ arbitrary and $b_1, b_2$ tending
to 0 along a countable set. Thus, it is enough to show the claim for the set $H=H(s,\thet, b_1, b_2)$ with some fixed parameters.
In view of the reduction, it is enough to prove power decay for $\Psi(\thb,\eps,\bz)$,
for almost every $\thb\in H$. As explained in Remark~\ref{rem1},
 it suffices to obtain \eqref{new-decay} for any fixed $\ell\in [s]$, with an exceptional set
in $\T^+$ of zero Hausdorff dimension. Finally, such a claim follows from one with an exceptional set of Hausdorff dimension less that $\beta>0$, for any
fixed $\beta$, and this is exactly the statement of Theorem~\ref{th:tech}.
\end{proof}

For the proof of Theorem~\ref{th:tech}
we again use a variant of the Erd\H{o}s-Kahane argument. The ``template'' is similar to that used in the proof Theorem~\ref{th:new1}, but the 
details are more technical. The argument is, essentially, a combination of the
relevant parts of \cite{ShSol16b} and \cite{Sol22}; see also \cite{SoSp23}. 

We can assume, without
loss of generality, that the undetermined eigenvalue is the last one, that is, $\ell=s$ in Theorem~\ref{th:tech}. 

Next we define the exceptional set, denoted $\Ek_\ell=\Ek_s$ in Theorem~\ref{th:tech}.
Let $H'$ be the projection of $H$ onto the first $2s-2$ coordinates of $\C^{2s}$. Fix $\thb'\in H'$ and
$\delta, \rho >0$. We first define the ``bad set'' at scale $N$:
\begin{eqnarray} \nonumber
E_{H,N}(\delta,\rho) & = & 
\Bigl\{\theta_{2s-1}\in \thet\cdot \T^+\subset \C:\ (\thb',\theta_{2s-1},\theta_{2s})\in H,\ \  \theta_{2s}= \ov\theta_{2s-1},\\
& & \left.\max_{\substack{\\[0.5ex] \tau:\ |\tau_{2s-1}| = \|\tau\|_\infty\\[0.5ex]1\le \|\tau\|_\infty < \vartheta}} \frac{1}{N}  
\left|\Bigl\{n\in [N]:\ \Big\|\sum_{j=1}^{2s} \theta_j^n \tau_j \Bigr\| < \rho\Bigr\}\right| >1-\delta \right\}. \label{excep10}
\end{eqnarray}
The exceptional set is given by
$$
\Ek_H(\delta,\rho)=\bigcap_{N_0=1}^\infty \bigcup_{N=N_0}^\infty E_{H,N}(\delta,\rho).
$$
Theorem~\ref{th:tech} will immediately follow from the next two propositions.

\begin{prop} \label{prop-excep0}
For any $\delta>0$, and $\rho\in (0,\half)$, there exists $\gam = \gam(H,\delta,\rho,\eps)$ such that 
$$
 \Psi(\thb,\eps,\bz) \le O_{H}(1)\cdot \|\bz\|_\infty^{-\gam}\ \ \mbox{for all}\ \bz\in \C^{2s},\ \mbox{with} \ \ 
|z_{2s-1}|=\|\bz\|_\infty \ge O_{\thb}(1).
$$
whenever 
$\thb=(\thb',\theta_{2s-1},\theta_{2s})\in H$ is such that 
$\theta_{2s-1}\notin \Ek_H(\delta,\rho)$. 
\end{prop}

\begin{sloppypar}
\begin{prop} \label{prop-dim0} For any  $\beta>0$ there 
exist  positive $\rho=\rho_H$ and $\delta=\delta_H$ such that $\Dh(\Ek_H(\delta,\rho))\le\beta$.
\end{prop}
\end{sloppypar}

\begin{proof}[Proof of Proposition~\ref{prop-excep0}]
Suppose that $\thb=(\thb',\theta_{2s-1},\theta_{2s})\in H$ and $\theta_{2s-1}\notin \Ek_H(\delta,\rho)$. This implies that 
there exists $N_0=N_0(\thb)\ge 1$ such that  $\theta_{2s-1}\not\in E_{H,N}(\delta,\rho)$ for all $N\ge N_0$.
Let $\bz\in \C^{2s}$ be such that ${\|\bz\|}_{_{\scriptstyle{\infty}}} \ge \vartheta^{N_0}$ and $|z_{2s-1}| = {\|\bz\|}_{_{\scriptstyle{\infty}}}$. 
As in \eqref{eq-tau}, let
$N = N(\bz)\ge 1$ be such that $\thet^{-N(\bz)}{\|\bz\|}_{_{\scriptstyle{\infty}}} \in [1,\thet)$, 
and $\tau\in \C^{2s}$ is such that $\tau_j = \theta_j^{-N(\bz)} z_j,\ j\in [2s]$, where $\bz = (z_1,\ldots,z_{2s})$.
Thus $|\tau_{2s-1}| = \|\tau\|_\infty$.
From the fact that $\theta_{2s-1}\not\in E_{H,N}(\delta,\rho)$ it follows that 
$$
\frac{1}{N}  
\left|\Bigl\{n\in [N]:\ \Big\|\sum_{j=1}^{2s} \theta_j^n \tau_j \Bigr\| \ge  \rho\Bigr\}\right| \ge \delta.
$$
Then by \eqref{imp22},
$$
\Psi(\thb,\eps,\bz) \le (1 - 2\pi\eps \rho^2)^{\lfloor\delta N\rfloor} \le (1 - 2\pi\eps \rho^2)^{\delta N-1}.
$$
By the definition of $N=N(\bz)$, we have 
$$
{\|\bz\|}_{_{\scriptstyle{\infty}}}\le \vartheta^{N+1}.
$$
We conclude that
$$
\Psi(\thb,\eps,\bz)= 
O_{\thet}(1)\cdot {\|\bz\|}_{_{\scriptstyle{\infty}}}^{-\gam},\ \ \mbox{with}\ \ {\|\bz\|}_{_{\scriptstyle{\infty}}} \ge \thet^{N_0(\thb)},
$$
for $\gam = -\delta\log (1 - 2\pi\eps \rho^2)/\log \vartheta$, and the proof is complete.
\end{proof}

All that remains is the proof of Proposition~\ref{prop-dim0}, which is the heart of the Erd\H{o}s-Kahane method.


\subsection{Beginning of the proof of Proposition~\ref{prop-dim0}}

We need the following

\begin{prop} \label{propa-EK0} Under the assumptions of the previous section and with its notation,
there exists a constant $\rho=\rho_H>0$ such that for any $N\in \N$ sufficiently large (depending on $H$), for any $\delta\in (0,\half)$,  the set $E_{H,N}(\delta,\rho)$ can be covered
by $\exp(O_{H}(\delta \log(1/\delta)N))$ disks of diameter $\vartheta^{-N}$.
\end{prop}

We first derive Proposition~\ref{prop-dim0}, assuming Proposition~\ref{propa-EK0}. By Proposition~\ref{propa-EK0},
$$
\Hk^\beta\left(\bigcup_{N=N_0}^\infty E_{H,N}(\delta,\rho)\right) \le \sum_{N=N_0}^\infty \exp(O_{H}(\delta \log(1/\delta)N))\cdot 
\vartheta^{-\beta N} \to 0, \ \ \mbox{as}\ N_0\to \infty,
$$
provided $\delta>0$ is so small that $\beta\log \vartheta >O_{H}(\delta \log(1/\delta))$. Thus $\Dh(\Ek_{H}(\delta,\rho))\le \beta$. \qed

\medskip


\begin{proof}[Proof of Proposition~\ref{propa-EK0}]
Recall that we have $\thb'=(\theta_1,\ldots,\theta_{d-2})\in H'$ fixed, where $d=2s$. For $\theta_{d-1}\in \C$ and $\theta_{d}=\ov\theta_{d-1}$, 
such that $(\thb',\theta_{d-1},\theta_{d})\in H$, together with $\tau$, as in the \eqref{excep10}, satisfying
$$
\tau_{2j} = \ov{\tau}_{2j-1},\ \ j\in [s],
$$
 we write
\beq \label{def-Kn}
\sum_{j=1}^{d} \tau_j \theta_j^n = K_n+\eps_n,\ \ n\ge 0,
\eeq
where $K_n\in \Z$ is the nearest integer to the expression in the left-hand side, so that 
$|\eps_n|\le \half$. Observe that $K_n$ depends on $\tau$ and $\theta_{d-1}$, although this is
not explicit in notation.

The next, ``algebraic'' step of the proof, starts by following the  scheme of \cite[Lemma 2.4]{Sol22}, with appropriate modifications.
Define $A_n^{(0)} = K_n$, $\wt A_n^{(0)} = K_n + \eps_n$, and then for all $n$ inductively: 
\beq\label{induc1}
A_n^{(j)} = A_{n+1}^{(j-1)} - \theta_j  A_{n}^{(j-1)};\ \ \ \ \ \wt A_n^{(j)} = \wt A_{n+1}^{(j-1)} - \theta_j  \wt A_{n}^{(j-1)},\ \ j=1,\ldots, d-2.
\eeq
It is easy to verify by induction that for all $j=0,\ldots, d-2$, holds
\beq \label{form00}
A_n^{(j)} = K_{n+j} + \sum_{k=1}^j (-1)^k \sig_k (\theta_1,\ldots,\theta_j) K_{n+j-k},\ \ n\in \N,
\eeq
where $\sig_k$ is an elementary symmetric polynomials (the sum of all distinct products of $k$ distinct variables). In particular, we have
\beq \label{form0}
A_n^{(d-2)} = K_{n+(d-2)} + \sum_{k=1}^{d-2} (-1)^k \sig_k (\theta_1,\ldots,\theta_j) K_{n+(d-2)-k},\ \ n\in \N.
\eeq
On the other hand, by induction in $j=0,\ldots, d-2$, we have
$$
\wt A_n^{(j)} = \sum_{i=j+1}^{d} \tau_i \prod_{k=1}^j (\theta_i - \theta_k) \theta_i^n,\ \ j=1,\ldots, d-2,\ \ n\in \N,
$$
hence for all $n\in \N$,
\begin{eqnarray} \nonumber
\wt A_n^{(d-2)} & = &  \tau_{d-1}  \prod_{k=1}^{d-2} (\theta_{d-1} - \theta_k) \theta_{d-1}^n+  \tau_{d}  \prod_{k=1}^{d-2} (\theta_{d} - \theta_k) \theta_{d}^n \\
 & =& 2\re\left[\tau_{d-1} \prod_{k=1}^{d-2} (\theta_{d-1} - \theta_k) \theta_{d-1}^n\right], \label{for1}
\end{eqnarray}
using that $\tau_d = \ov\tau_{d-1}$  and $\theta_{2j} = \ov\theta_{2j-1}$ for $j\in [2s]$.

We have $|\theta_j| = \thet$ and $\left|\wt A_n^{(0)} - A_n^{(0)} \right|\le|\eps_n|$, and then by induction in $j$, in view of \eqref{induc1},
\beq
\left|\wt A_n^{(j)} - A_n^{(j)} \right| \le \sum_{k=0}^j \thet^{j-k} |\eps_{n+j}| 
                        < (1+\thet)^j \max\left\{|\eps_n|,\ldots,|\eps_{n+j}|\right\},\ \ j\le d-2,\ \ n\in \N.\label{bound1}
\eeq

Next we need a lemma from \cite{ShSol16} which provides a solution to the system of equations
\begin{eqnarray*}
		  \re(w_0) & = & x_0   \\
		 \re(\theta w_0)  &= & x_1  \\
		  \re(\theta^2 w_0) & = & x_2   \\
		  \re(\theta^3 w_0) & = & x_3,
\end{eqnarray*}
with the unknowns $\theta,w_0\in \C$ satisfying some a priori bounds, and with the right-hand sides satisfying appropriate constraints (note that the system is overdetermined, since we are given $|\theta| = \thet$).
In fact, the solution is given by explicit algebraic functions; however, their precise form is not important.
In our situation we will have (for a fixed $n$):
\beq \label{con3}
\theta = \theta_{d-1},\ \ w_0 = 2\tau_{d-1} \prod_{k=1}^{d-2} (\theta_{d-1} - \theta_k) \theta_{d-1}^n,\ \ x_j = \wt A_{n+j}^{(d-2)}= \re(\theta^j w_0),\ j=0,\ldots,3.
\eeq
Then, replacing $\wt A_{n+j}^{(d-2)}$ by $A_{n+j}^{(d-2)}$, we will obtain a good enough approximation of $\theta_{d-1}$ in terms
of $K_n$'s and $\thb'$.

Given $\thet >1$ and $b_1>0$, let
\beq \label{defH}
\Ak_{\thet, b_1} := \bigl\{ \theta \in \C:\ |\theta| = \thet,\ \im(\theta) \ge b_1\bigr\}.
\eeq 
Further, for $R_0>0$, let
	$$
	V_{R_0} =\bigl\{\bx = (x_0,x_1,x_2,x_3)\in \R^4:\ \exists\, w_0\in \C,\ |w_0|\ge R_0,\ \exists\, \theta\in \Ak_{\thet, b_1},\ x_j = \re(\theta^j w_0),\ 0\le 
	j\le 3\bigr\}.
	$$
	Denote by $\Nk_\eps (V)$ the open $\eps$-neighborhood of $V$ in the 
	$\ell^\infty$ metric.

	\begin{lemma}[{variant of \cite[Lemma 2.2]{ShSol16}}] \label{lem-tech} Given $r = r(\vartheta, b_1)>0$, 
		there exist $R_0>0$ and $C_1>0$ depending only on $\vartheta$ and $b_1$, such 
		that
			there exist continuously differentiable functions 
			$$F:\,\Nk_r (V_{R_0})\to \{\theta\in \C:\ |\theta|>1, \ \im(\theta)>0\}\ \ 
			\mbox{and}\ \  G:\,\Nk_r(V_{R_0})\to \R,$$  with $\theta = 
			F(\bx)$ and $y_3 = G(\bx)$ being the unique solutions to the system of equations
			\beq \label{equs}
			\re(\theta^{j-3} (x_3+iy_3))=x_j,\ \ \ 0\le j \le 2,
			\eeq
			given an arbitrary $\bx = (x_0,x_1,x_2,x_3)\in V_{R_0}$.
			Moreover,
			$$
			\left|\frac{\partial G}{\partial x_j}\right| \le C_1,\ \ \ 0\le 
			j\le 3,\ \ \ \mbox{on}\ \ \Nk_r(V_{R_0}).
			$$
	\end{lemma}

The statement of  \cite[Lemma 2.2]{ShSol16} is with $r=1$; however, the same proof works for any $r>0$ depending on $\thet$ and $b_1$; we
let
$$r= (1/2)(1+\thet)^{d-2}.$$
Observe that in \eqref{con3} we have $\theta = \theta_{d-1}\in \Ak_{\thet,b_1}$  by the assumption that $\thb\in H$. Furthermore, by \eqref{con3},
\beq \label{esta1}
|w_0| \ge 2 b_2^{d-2} \thet^n \ge R_0\ \ \ \mbox{for}\ n\ge n_1=n_1(H,\beta),
\eeq
where we used that $|\tau_{d-1}| = \|\tau\|_\infty \ge 1$.
Thus,
$$
\bx = \left(\wt A_{n}^{(d-2)}, \ldots, \wt A_{n+3}^{(d-2)}\right) \in  V_{R_0}\ \ \ \ \mbox{for}\ n\ge n_1,
$$
in view of \eqref{for1} and \eqref{esta1}. Now we can apply Lemma~\ref{lem-tech} with our data from \eqref{con3} for $n \ge n_1$. 
Denote
$$
\wt Y_{n} := \im\left(2\tau_{d-1} \prod_{k=1}^{d-2} (\theta_{d-1} - \theta_k) \theta_{d-1}^n\right),\ \ n\in \N.
$$
Applying Lemma~\ref{lem-tech}, we obtain
$$
\wt Y_{n+3} = G\left(\wt A_n^{(d-2)}, \ldots, \wt A_{n+3}^{(d-2)}\right),\ \ n\ge n_1,
$$
so that
\beq \label{form2}
\wt A_{n+3}^{(d-2)} + i \wt Y_{n+3} =  2\tau_{d-1} \prod_{k=1}^{d-2} (\theta_{d-1} - \theta_k) \theta_{d-1}^{n+3},\ \ n\ge n_1.
\eeq

Next we would like to apply the function $G$ to
perturbed data. By \eqref{bound1} we have for all $n\in \N$:
\beq \label{bound101}
\left|A_{n}^{(d-2)} - \wt A_{n}^{(d-2)}\right|  < (1+\thet)^{d-2}\max\{|\eps_n|,\ldots,|\eps_{n+d-2}|\} \le
(1/2) (1+\thet)^{d-2} = r,
\eeq
hence
$\bx' = \left(A_{n}^{(d-2)}, \ldots, A_{n+3}^{(d-2)}\right)$ satisfies $\|\bx-\bx'\|_\infty < r$, and so $\bx'\in  \Nk_r(V_{R_0})$ for $n\ge n_1$.
Thus,
$$
Y_{n+3} := G\left(A_n^{(d-2)}, \ldots, A_{n+3}^{(d-2)}\right),\ \ n\ge n_1,
$$
is well-defined.
Writing $Y_{n+3} - \wt Y_{n+3}$ as a sum of path integrals of the partial derivatives of $G$, we obtain from Lemma~\ref{lem-tech} and \eqref{bound101} that
\beq \label{bound4}
|Y_{n+3} - \wt Y_{n+3}| \le 4C_1 (1+\thet)^{d-2} \max\{|\eps_n|,\ldots, |\eps_{n+d+1}|\},\ \ n\ge n_1
\eeq
(the factor 4 comes from the estimate $\|\bx\|_1 \le \|\bx\|_\infty$ on $\R^4$). 

\subsection{Approximation}
By \eqref{form2},
\begin{eqnarray}
\theta_{d-1} = \frac{\wt A_{n+4}^{(d-2)} + i \wt Y_{n+4}}{\wt A_{n+3}^{(d-2)} + i \wt Y_{n+3}} & \approx & \nonumber
\frac{ A_{n+4}^{(d-2)} + i Y_{n+4}}{ A_{n+3}^{(d-2)} + i Y_{n+3}} \\[1.2ex]
& = & \frac{A_{n+4}^{(d-2)} + i G\bigr(A_{n+1}^{(d-2)}, \ldots, A_{n+4}^{(d-2)}\bigr)}{A_{n+3}^{(d-2)} + i G\bigl(A_{n}^{(d-2)}, \ldots, A_{n+3}^{(d-2)}\bigr)} \nonumber \\[1.2ex]
& =: & \Phi_{\thb'}(K_n,\ldots,K_{n+d+2}),\label{erdo1} 
\end{eqnarray}
for some  function $\Phi_{\thb'}$ of $(d+3)$ variables, depending on the parameters $\thb'$ that have been fixed, in view of \eqref{form00}. Thus, given $K_n,\ldots,K_{n+d+2}$, we are
able to get a good approximation of $\theta_{d-1}$. More precisely, we have

\begin{lemma} \label{lem:erd101}
There exist $C_2 = C_2(H)$ and $n_2 = n_2(H)\ge n_1(H)$, such that 
\beq\label{bound404}
\left| \theta_{d-1} - \Phi_{\thb'}(K_n,\ldots,K_{n+d+2})\right| \le C_2\cdot \vartheta^{-n} \max\{|\eps_n|,\ldots, |\eps_{n+d+2}|\},\ \ \ n\ge n_2.
\eeq
\end{lemma}

\begin{proof} Denote
$$
B_n:= A_n^{(d-2)} + iY_n\ \ \ \mbox{and}\ \ \ \wt B_n := \wt A_n^{(d-2)} + i\wt Y_n.
$$
Then \eqref{erdo1} implies
\beq\label{eq303}
\bigl| \theta_{d-1} - \Phi_{\thb'}(K_n,\ldots,K_{n+d+2})\bigr|  = \left| \theta_{d-1} -\frac{B_{n+4}}{B_{n+3}}\right|=
  \left| \frac{\wt B_{n+4}}{\wt B_{n+3}} - \frac{B_{n+4}}{B_{n+3}}\right|.
\eeq
Observe that
\beq \label{bound202}
|B_{n+3} - \wt B_{n+3}| \le (1+4C_1)\cdot(1+\thet)^{d-2} \max\{|\eps_n|,\ldots, |\eps_{n+d+1}|\},\ \ n\ge n_1,
\eeq
by \eqref{bound101} and \eqref{bound4}, whereas
\beq \label{bound5}
2 b_2^{d-2} \thet^{n+3} \le |\wt B_{n+3}| = \bigl|\wt A_{n+3}^{(d-2)} + i \wt Y_{n+3}\bigr| \le 2\cdot(2\thet)^{d-1} \thet^{n+3},\ \ n\ge n_1,
\eeq
in view of \eqref{form2}. Now the desired estimate follows easily.
\end{proof}

Analogously, given $K_n,\ldots,K_{n+d+2}$, we can make a good ``prediction'' of $K_{n+d+3}$. Indeed,
$$
\theta_{d-1}= \frac{\wt A_{n+5}^{(d-2)} + i \wt Y_{n+4}}{\wt A_{n+4}^{(d-2)} + i \wt Y_{n+4}}  =  \frac{\wt A_{n+4}^{(d-2)} + i \wt Y_{n+4}}{\wt A_{n+3}^{(d-2)} + i \wt Y_{n+3}} 
$$
yields
\begin{eqnarray*}
\wt A_{n+5}^{(d-2)} = \re\left[ \frac{(\wt A_{n+4}^{(d-2)} + i \wt Y_{n+4})^2}{\wt A_{n+3}^{(d-2)} + i \wt Y_{n+3}}\right],
\end{eqnarray*}
hence
\begin{eqnarray} \label{eq311}
A_{n+5}^{(d-2)} \approx \re\left[ \frac{( A_{n+4}^{(d-2)} + i Y_{n+4})^2}{ A_{n+3}^{(d-2)} + i Y_{n+3}}\right] & = &
\re\left[ \frac{\Bigl(A_{n+4}^{(d-2)} + i G\bigl(A_{n+1}^{(d-2)}, \ldots, A_{n+4}^{(d-2)}\bigr)\Bigr)^2}{A_{n+3}^{(d-2)} + i G\bigl(A_{n}^{(d-2)}, \ldots, A_{n+3}^{(d-2)}\bigr)} \right], \\[1.2ex]
& =: & \Psi_{\thb'}(K_n,\ldots,K_{n+d+2}), \nonumber
\end{eqnarray}
for another  function, $\Psi_{\thb'}$, of $(d+3)$ variables, depending on the parameters $\thb'$. In view of \eqref{form0}, this yields
\begin{eqnarray} \nonumber
K_{n+d+3} & =           & A_{n+5}^{(d-2)} - \sum_{k=1}^{d-2} (-1)^k \sig_k (\theta_1,\ldots,\theta_j) K_{n+(d+3)-k} \\
 & \approx & \Psi_{\thb'}(K_n,\ldots,K_{n+d+2}) - \sum_{k=1}^{d-2} (-1)^k \sig_k (\theta_1,\ldots,\theta_j) K_{n+(d+3)-k} \label{erdo2} \\
                  & =:          & \Xi_{\thb'}(K_n,\ldots,K_{n+d+2}), \nonumber
\end{eqnarray}
for yet another function, $\Xi_{\thb'}$, of $(d+3)$ variables, depending on the parameters $\thb'$. 
The next lemma makes  this approximation precise.

\begin{lemma} \label{lem:erd1}
There exist $C_3 = C_3(H)$ and $n_3 = n_3(H)\ge n_2(H)$, such that 
$$
|K_{n+d+3} - \Xi_{\thb'}(K_n,\ldots,K_{n+d+2})| \le C_3\cdot  \max\{|\eps_n|,\ldots, |\eps_{n+d+3}|\},\ \ \ n\ge n_3.
$$
\end{lemma}

\begin{proof}  We have by \eqref{eq311}:
\begin{eqnarray*}
|K_{n+d+3} - \Xi_{\thb'}(K_n,\ldots,K_{n+d+2})| & = & \left| A_{n+5}^{(d-2)} - \re\bigl(B_{n+4}^2/B_{n+3}\bigr) \right| \\
& \le & \left| B_{n+5} - \frac{B_{n+4}^2}{B_{n+3}}\right| = |B_{n+4}|\cdot \left| \frac{B_{n+5}}{B_{n+4}} - \frac{B_{n+4}}{B_{n+3}}\right|,
\end{eqnarray*}
and the desired estimate follows from \eqref{eq303}, \eqref{bound404}, and \eqref{bound5}.
\end{proof}

\subsection{Conclusion of the proof}

Let
\beq \label{def-rho}
\rho:= (2C_3)^{-1}\ \ \ \mbox{and}\ \ \ M:= 2C_3 + 1.
\eeq
Lemma~\ref{lem:erd1}  immediately implies the following

\begin{lemma} \label{lem:erd2}
Consider an arbitrary $\theta_{d-1} \in \C$, such that $(\thb',\theta_{d-1},\ov\theta_{d-1})\in H$,  and $\tau = (\tau_1,\ldots,\tau_d)\in \C^{2s}$, 
such that  $\tau_{2j} = \ov\tau_{2j-1}$ for
$j\in [s]$, with $1\le |\tau_{d-1}|= \|\tau\|_\infty\le \thet$. Define the corresponding sequences $K_n, \eps_n$ by \eqref{def-Kn}.
Then the following hold, independent of $\theta_{d-1}$:
\begin{enumerate}
\item[{\bf (i)}] for any $n\ge n_3(H)$, such that $\max\{|\eps_n|,\ldots, |\eps_{n+d+3}|\}<\rho$,
 the number $K_{n+d+3}$ is uniquely determined by $(K_j)_{j=n}^{n+d+2}$.
\item[{\bf (ii)}] for any $n\ge n_3(H)$  there are at most $M$ choices of $K_{n+d+3}$ given $(K_j)_{j=n}^{n+d+2}$.
\end{enumerate}
\end{lemma}

Now we can finish the proof of Proposition~\ref{propa-EK0}. Suppose that $N\ge n_3(H)$. 
Further, fix $\theta_{d-1}\in E_{H,N}(\delta,\rho)$, with $\rho$ from \eqref{def-rho} and $\delta>0$ arbitrary. By definition, there exists
$\tau = (\tau_1,\ldots,\tau_d)\in \C^d$, such that $\tau_{2j} = 
\ov\tau_{2j-1}\in \C$ for $j\in [s]$, with $1\le |\tau_{d-1}|= \|\tau\|_\infty\le \vartheta$, satisfying
$$
\Bigl|\Bigl\{n\in [N]:\ \Big\|\sum_{j=1}^{d} \tau_j \theta_j^n \Bigr\| < \rho\Bigr\}\Bigr| \ge (1-\delta)N.
$$
Since $\thb\in H$, there are $O_H(1)$ choices for the initial part of the sequence $K_1,\ldots, K_{n_3}$. The set $J := \{n\in [N]:\ |\eps_n|\ge \rho\}$ has cardinality at most 
$\lfloor \delta N \rfloor$, by construction. In view of Lemma~\ref{lem:erd2}, given $J$, there are at most $O_{H}(M^{\delta (d+4)N})$ choices for the sequence $K_1,\ldots,
K_{N+d+2}$. By \eqref{bound404}, the minimal number of disks of radius $O_{H}(\thet^{-N})$ needed to cover $E_{H,N}(\delta,\rho)$, is at most
$$
O_{H}(M^{\delta (d+4)N})\cdot {N\choose \lfloor \delta N \rfloor} = \exp(O_{H}(\delta \log(1/\delta)N)),
$$
as desired. This completes the proof of Proposition~\ref{propa-EK0}, and hence also of Proposition~\ref{prop-dim0}, Theorem~\ref{th:tech}, and
Theorem~\ref{thm2}.
\end{proof}

\subsection*{Acknowledgements} I am grateful to Pablo Shmerkin for many helpful discussions, especially on the issues related to his paper
\cite{CoShm24} with Emilio Corso.

\bibliographystyle{plain}
\bibliography{nonunif}

\begin{thebibliography}{10}

\bibitem{AHW21}
Amir Algom, Federico Rodriguez~Hertz, and Zhiren Wang.
\newblock Pointwise normality and {F}ourier decay for self-conformal measures.
\newblock {\em Adv. Math.}, 393:Paper No. 108096, 72, 2021.

\bibitem{ARW23}
Amir {Algom}, Federico {Rodriguez Hertz}, and Zhiren {Wang}.
\newblock {Polynomial Fourier decay and a cocycle version of Dolgopyat's method
  for self conformal measures}.
\newblock {\em arXiv e-prints}, page arXiv:2306.01275, June 2023.

\bibitem{ARW24}
Amir {Algom}, Federico {Rodriguez Hertz}, and Zhiren {Wang}.
\newblock {Spectral gaps and Fourier decay for self-conformal measures in the
  plane}.
\newblock {\em arXiv e-prints}, page arXiv:2407.11688, July 2024.

\bibitem{BaBa24}
Simon {Baker} and Amlan {Banaji}.
\newblock {Polynomial Fourier decay for fractal measures and their
  pushforwards}.
\newblock {\em arXiv e-prints}, page arXiv:2401.01241, January 2024.

\bibitem{BaKhaSa24}
Simon {Baker}, Osama {Khalil}, and Tuomas {Sahlsten}.
\newblock {Fourier Decay from $L^2$-Flattening}.
\newblock {\em arXiv e-prints}, page arXiv:2407.16699, July 2024.

\bibitem{Bremont}
Julien Br\'{e}mont.
\newblock Self-similar measures and the {R}ajchman property.
\newblock {\em Ann. H. Lebesgue}, 4:973--1004, 2021.

\bibitem{BroMonSid04}
Dave Broomhead, James Montaldi, and Nikita Sidorov.
\newblock Golden gaskets: variations on the {S}ierpi\'nski sieve.
\newblock {\em Nonlinearity}, 17(4):1455--1480, 2004.

\bibitem{BuSo14}
Alexander~I. Bufetov and Boris Solomyak.
\newblock On the modulus of continuity for spectral measures in substitution
  dynamics.
\newblock {\em Adv. Math.}, 260:84--129, 2014.

\bibitem{CoShm24}
Emilio {Corso} and Pablo {Shmerkin}.
\newblock {Dynamical self-similarity, $L^{q}$-dimensions and Furstenberg
  slicing in $\mathbb{R}^d$}.
\newblock {\em arXiv e-prints}, September 2024.
\newblock arXiv:2409.04608.

\bibitem{Dai}
Xin-Rong Dai.
\newblock When does a {B}ernoulli convolution admit a spectrum?
\newblock {\em Adv. Math.}, 231(3-4):1681--1693, 2012.

\bibitem{DFW}
Xin-Rong Dai, De-Jun Feng, and Yang Wang.
\newblock Refinable functions with non-integer dilations.
\newblock {\em J. Funct. Anal.}, 250(1):1--20, 2007.

\bibitem{Erd1}
Paul Erd\H{o}s.
\newblock On a family of symmetric {B}ernoulli convolutions.
\newblock {\em Amer. J. Math.}, 61:974--976, 1939.

\bibitem{Erd2}
Paul Erd\H{o}s.
\newblock On the smoothness properties of a family of {B}ernoulli convolutions.
\newblock {\em Amer. J. Math.}, 62:180--186, 1940.

\bibitem{GaoMa17}
Xiang Gao and Jihua Ma.
\newblock Decay rate of {F}ourier transforms of some self-similar measures.
\newblock {\em Acta Math. Sci. Ser. B (Engl. Ed.)}, 37(6):1607--1618, 2017.

\bibitem{Hochman14}
Michael Hochman.
\newblock On self-similar sets with overlaps and inverse theorems for entropy.
\newblock {\em Ann. of Math. (2)}, 180(2):773--822, 2014.

\bibitem{Hochman}
Michael {Hochman}.
\newblock {On self-similar sets with overlaps and inverse theorems for entropy
  in $\mathbb{R}^d$}.
\newblock {\em arXiv e-prints, Memoirs of the AMS, to appear}, March 2015.
\newblock arXiv:1503.09043.

\bibitem{Jordan05}
Thomas Jordan.
\newblock Dimension of fat {S}ierpi\'nski gaskets.
\newblock {\em Real Anal. Exchange}, 31(1):97--110, 2005/06.

\bibitem{JorPoli06}
Thomas Jordan and Mark Pollicott.
\newblock Properties of measures supported on fat {S}ierpinski carpets.
\newblock {\em Ergodic Theory Dynam. Systems}, 26(3):739--754, 2006.

\bibitem{kahane}
J.-P. Kahane.
\newblock Sur la distribution de certaines s\'{e}ries al\'{e}atoires.
\newblock {\em Bull. Soc. Math. France, M\'{e}m. No. 25, Soc. Math. France
  Paris}, pages 119--122, 1971.

\bibitem{Kershner}
Richard Kershner.
\newblock On {S}ingular {F}ourier-{S}tieltjes {T}ransforms.
\newblock {\em Amer. J. Math.}, 58(2):450--452, 1936.

\bibitem{Kittle}
Samuel Kittle.
\newblock Absolutely continuous self-similar measures with exponential
  separation.
\newblock {\em Ann. Sci. \'Ec. Norm. Sup\'er. (4)}, 57(4):1191--1231, 2024.

\bibitem{KiKog24}
Samuel {Kittle} and Constantin {Kogler}.
\newblock {On absolute continuity of inhomogeneous and contracting on average
  self-similar measures}.
\newblock {\em arXiv e-prints}, page arXiv:2409.18936, September 2024.

\bibitem{LPS25}
Ga{\'e}tan {Leclerc}, Sampo {Paukkonen}, and Tuomas {Sahlsten}.
\newblock {Fourier Dimension in $C^{1+\alpha}$ Parabolic Dynamics}.
\newblock {\em arXiv e-prints}, page arXiv:2505.15468, May 2025.

\bibitem{LS1}
Jialun Li and Tuomas Sahlsten.
\newblock Trigonometric series and self-similar sets.
\newblock {\em J. Eur. Math. Soc. (JEMS)}, 24(1):341--368, 2022.

\bibitem{LinVar16}
Elon Lindenstrauss and P\'eter~P. Varj\'u.
\newblock Random walks in the group of {E}uclidean isometries and self-similar
  measures.
\newblock {\em Duke Math. J.}, 165(6):1061--1127, 2016.

\bibitem{Mattila:Fourier-book}
Pertti Mattila.
\newblock {\em Fourier analysis and {H}ausdorff dimension}, volume 150 of {\em
  Cambridge Studies in Advanced Mathematics}.
\newblock Cambridge University Press, Cambridge, 2015.

\bibitem{PS00}
Yuval Peres and Wilhelm Schlag.
\newblock Smoothness of projections, {B}ernoulli convolutions, and the
  dimension of exceptions.
\newblock {\em Duke Math. J.}, 102(2):193--251, 2000.

\bibitem{PSS00}
Yuval Peres, Wilhelm Schlag, and Boris Solomyak.
\newblock Sixty years of {B}ernoulli convolutions.
\newblock In {\em Fractal geometry and stochastics, {II}
  ({G}reifswald/{K}oserow, 1998)}, volume~46 of {\em Progr. Probab.}, pages
  39--65. Birkh\"auser, Basel, 2000.

\bibitem{PeresSolomyak96}
Yuval Peres and Boris Solomyak.
\newblock Absolute continuity of {B}ernoulli convolutions, a simple proof.
\newblock {\em Math. Res. Lett.}, 3(2):231--239, 1996.

\bibitem{SSS}
Santiago Saglietti, Pablo Shmerkin, and Boris Solomyak.
\newblock Absolute continuity of non-homogeneous self-similar measures.
\newblock {\em Adv. Math.}, 335:60--110, 2018.

\bibitem{Sahlsten23}
Tuomas {Sahlsten}.
\newblock {Fourier transforms and iterated function systems}.
\newblock {\em arXiv e-prints}, page arXiv:2311.00585, November 2023.

\bibitem{Salem}
Raphael Salem.
\newblock Sets of uniqueness and sets of multiplicity.
\newblock {\em Trans. Amer. Math. Soc.}, 54:218--228, 1943.

\bibitem{Shmerkin14}
Pablo Shmerkin.
\newblock On the exceptional set for absolute continuity of {B}ernoulli
  convolutions.
\newblock {\em Geom. Funct. Anal.}, 24(3):946--958, 2014.

\bibitem{ShSol16b}
Pablo Shmerkin and Boris Solomyak.
\newblock Absolute continuity of complex {B}ernoulli convolutions.
\newblock {\em Math. Proc. Cambridge Philos. Soc.}, 161(3):435--453, 2016.

\bibitem{ShSol16}
Pablo Shmerkin and Boris Solomyak.
\newblock Absolute continuity of self-similar measures, their projections and
  convolutions.
\newblock {\em Trans. Amer. Math. Soc.}, 368(7):5125--5151, 2016.

\bibitem{SiSol02}
K\'aroly Simon and Boris Solomyak.
\newblock On the dimension of self-similar sets.
\newblock {\em Fractals}, 10(1):59--65, 2002.

\bibitem{Solomyak95}
Boris Solomyak.
\newblock On the random series {$\sum\pm\lambda\sp n$} (an {E}rd{\H o}s
  problem).
\newblock {\em Ann. of Math. (2)}, 142(3):611--625, 1995.

\bibitem{Sol_Fourier}
Boris Solomyak.
\newblock Fourier decay for self-similar measures.
\newblock {\em Proc. Amer. Math. Soc.}, 149(8):3277--3291, 2021.

\bibitem{Sol22}
Boris Solomyak.
\newblock Fourier decay for homogeneous self-affine measures.
\newblock {\em J. Fractal Geom.}, 9(1-2):193--206, 2022.

\bibitem{SoSp23}
Boris {Solomyak} and Adam {{\'S}piewak}.
\newblock {Absolute continuity of self-similar measures on the plane}.
\newblock {\em arXiv e-prints}, page arXiv:2301.10620, January 2023.

\bibitem{Streck23}
Lauritz {Streck}.
\newblock {On absolute continuity and maximal Garsia entropy for self-similar
  measures with algebraic contraction ratio}.
\newblock {\em arXiv e-prints}, page arXiv:2303.07785, March 2023.

\bibitem{Varju1}
P\'{e}ter~P. Varj\'{u}.
\newblock Absolute continuity of {B}ernoulli convolutions for algebraic
  parameters.
\newblock {\em J. Amer. Math. Soc.}, 32(2):351--397, 2019.

\bibitem{VarYu}
P\'eter~P. Varj\'u and Han Yu.
\newblock Fourier decay of self-similar measures and self-similar sets of
  uniqueness.
\newblock {\em Anal. PDE}, 15(3):843--858, 2022.

\end{thebibliography}

\end{document}